\documentclass[10pt, twoside, notitlepage]{amsart}

\usepackage[utf8]{inputenc}
\usepackage{color,xcolor}
\usepackage{amsmath}
\usepackage{amssymb}
\usepackage{amsthm}
\usepackage{geometry}
\usepackage{graphicx}
\usepackage{esint}
\usepackage[ocgcolorlinks,linkcolor=blue]{hyperref}

\theoremstyle{plain}
\newtheorem{thm}{Theorem}
\newtheorem{prop}{Proposition}[section]
\newtheorem{lem}[prop]{Lemma}
\newtheorem{cor}[prop]{Corollary}

\newtheorem{rmk}[prop]{Remark}

\newcommand {\R} {\mathbb{R}} \newcommand {\Z} {\mathbb{Z}}
 \newcommand {\N} {\mathbb{N}}
 
\newcommand {\p} {\partial}

\newcommand {\D} {\Delta}

\newcommand {\supp} {\text{supp}}

\DeclareMathOperator {\dist} {dist}

\allowdisplaybreaks

\title{Discrete Carleman estimates and three balls inequalities}

\author[A. Fern\'andez-Bertolin]{Aingeru Fern\'andez-Bertolin}
\address{Facultad de Ciencia y Tecnología, Universidad del País Vasco /Euskal Herriko Unibertsitatea (UPV/EHU), Departamento de Matemáticas, UPV/EHU, Apartado 644, 48080 Bilbao, Spain}
\email{aingeru.fernandez@ehu.eus}
\author[L. Roncal]{Luz Roncal}
\address{BCAM - Basque Center for Applied Mathematics \\
48009 Bilbao, Spain and Ikerbasque, Basque Foundation for Science, 48011 Bilbao, Spain}
\email{lroncal@bcamath.org }
\author[A. R\"uland]{Angkana R\"uland}
\address{Max-Planck Institute for Mathematics in the Sciences, Inselstraße 22, 04103 Leipzig, Germany}
\email{rueland@mis.mpg.de}
\curraddr{Ruprecht-Karls-Universität Heidelberg, Institut für Angewandte Mathematik, Im Neuenheimer Feld 205, 69120 Heidelberg, Germany}
\email{Angkana.Rueland@uni-heidelberg.de}

\author[D. Stan]{Diana Stan}
\address{University of Cantabria, Department of Mathematics, Statistics and Computation,
Avd. Los Castros 44, 39005 Santander, Spain}
\email{diana.stan@unican.es}

\date{\today}


\begin{document}

\maketitle

\begin{abstract}
We prove logarithmic convexity estimates and three balls inequalities for discrete magnetic Schrödinger operators. These quantitatively connect the discrete setting in which the unique continuation property fails and the continuum setting in which the unique continuation property is known to hold under suitable regularity assumptions. As a key auxiliary result which might be of independent interest we present a Carleman estimate for these discrete operators.
\end{abstract}

\section{Introduction}
\label{sec:intro}
In this article, we provide robust quantitative unique continuation results for discrete magnetic Schrödinger operators $P_h$ of the form
\begin{equation}
\label{eq:Schroedinger}
P_h f(n):=  h^{-2}\D_d f(n) + h^{-1}\sum\limits_{j=1}^{d} B_j(n) D_{+,j}^h f(n) + V(n) f(n),
\end{equation}
where $f: (h\Z)^d \rightarrow \R$, $D_{\pm,j}^h f(n):= \pm (f(n \pm h e_j)- f(n))$ denotes the (unscaled) left/right difference operator on scale $h$, $B: (h\Z)^d \rightarrow \R^d$ is a (uniformly in $h$) bounded tensor field, modelling, for instance, magnetic interactions and where the potential $V: (h \Z)^d \rightarrow \R$ is assumed to be uniformly bounded (independently of $h$). The operator $\D_{d}:=  \sum\limits_{j=1}^d \D_{d,j,h} = \sum\limits_{j=1}^d D_{-,j}^h  D_{+,j}^h$ is the (not normalized) discrete  Laplacian on the lattice $(h \Z)^d$.

The operators considered in \eqref{eq:Schroedinger} correspond to discrete versions of the continuous magnetic Schrödinger operator. While many features of the continuous and the discrete operators are shared, if correspondingly adapted (e.g. regularity estimates), there are striking differences in the validity of the unique continuation property in these settings. In fact, even for the case of the model operator, the discrete Laplacian, it is well-known that while in the \emph{continuum} the (weak) unique continuation property holds as a direct consequence of the analyticity of the solutions, this fails in general in the \emph{discrete} setting \cite{GM14}. Indeed, in \cite{GM14} the authors show that it is possible to construct non-trivial harmonic polynomials vanishing on a large, prescribed square. In spite of these differences, it is expected that as the lattice spacing decreases, $h \rightarrow 0$, the properties of continuous harmonic functions are recovered. That this is in fact the case for the setting of the discrete Laplacian was proved in \cite{GM14, GM13, LM15}, where propagation of smallness estimates with correction terms were proved for the discrete Laplacian. For similar phenomena for related operators we refer to \cite{BFV17, JLMP18} and the references therein.

Most of the cited propagation of smallness results from the literature however strongly relied on the specific properties of the constant coefficient  Laplacian, e.g. by using methods from complex analysis. It is the purpose of this article to provide quantitative unique continuation estimates and three spheres inequalities for a large class of Schrödinger operators by means of robust Carleman estimates. We emphasize that in addition to the intrinsic interest in the quantitative unique continuation properties of discrete elliptic equations, important applications of these quantitative unique continuation estimates involve inverse and control theoretic problems (see for instance \cite{BHJ10, E11}).

\subsection{Main results}
Let us describe our main results. As a first main result, we seek to prove a discrete analogue (with correction term) of a logarithmic convexity inequality. More precisely, for $u$ with $P_h u = 0$ and $P_h$ being the Schrödinger operator from \eqref{eq:Schroedinger} we provide the following bounds:

\begin{thm}
\label{thm:log_conv}
There exist constants $h_0,\delta_0\in(0,1)$, $c_1, c_2 >0$ and $ \tau_0>1$ such that for all $\tau \in (\tau_0, \delta_0 h^{-1})$, $h \in (0,h_0)$ and $u: (h\Z)^d \rightarrow \R$ with $P_h u = 0$ in $B_4$ it holds
\begin{equation}
\label{eq:log_conv}
\|u\|_{L^2(B_1)} \leq C(e^{c_1 \tau} \|u\|_{L^2(B_{1/2})} + e^{- c_2 \tau}\|u\|_{L^2(B_2)}).
\end{equation}
Here for $r > 0$ we define $B_{r}=B_r(0) \cap (h \Z)^d$, with $h\in (0,h_0)$ denoting the lattice spacing, and all $L^2$ norms are $L^2$ norms on the lattice $(h \Z)^d$.
\end{thm}

Due to the restriction on the upper bound of $\tau \leq \delta_0 h^{-1}$, this logarithmic convexity estimate does not immediately yield a three balls inequality as in the continuum. It however implies a three balls estimate with a corresponding correction term:

\begin{thm}
\label{thm:3spheres}
There exist $\alpha \in (0,1)$, $c_0>0$, $h_0\in(0,1)$ and $C>1$ such that for $h\in (0,h_0)$ and $u: (h\Z)^d \rightarrow \R$ with $P_h u = 0$ in $B_4$ we have
\begin{equation}
\label{eq:3_balls}
\|u\|_{L^2(B_1)} \leq C(\|u\|_{L^2(B_{1/2})}^{\alpha}\|u\|_{L^2(B_2)}^{1-\alpha} + e^{- c_0 h^{-1}}\|u\|_{L^2(B_2)}).
\end{equation}
\end{thm}

This estimate thus quantitatively connects the discrete situation in which the unique continuation property fails to its continuous counterpart. It provides quantitative evidence of the fact that as $h \rightarrow 0$, the propagation of smallness property of the associated elliptic operator is recovered. We remark that the scaling behaviour of the form $e^{-c_0 h^{-1}}$ in $h \in (0,h_0)$ had earlier been proven for the special case of the Laplacian (see \cite[Theorem 1]{GM14}) and is known to be optimal (see the discussion in \cite[Section 4]{GM14}). A similar, asymptotically optimal three balls estimate with Gaussian weight and with error terms is given in \cite{LM15}, see also \cite[Corollary 1.14]{LM15}.  

We remark that our results (and arguments) remain valid if instead of the differential equation \eqref{eq:Schroedinger} we consider the differential inequality
\begin{equation*}
|h^{-2}\D_{d} f(n)| \leq C\Big( h^{-1}\sum\limits_{j=1}^{d} |D_{+,j}^h f(n)| + |f(n)|  \Big) \mbox{ for some } 0<C<\infty.
\end{equation*}

Further, it is possible to deduce propagation of smallness estimates for some controlled $h$-dependent growth of $V$ and $B_{j}$ (see Remark \ref{rmk:singular_potentials}) which however, of course, do not pass to the limit as $h\rightarrow 0$.

\subsection{Main ideas}

Similarly as in \cite{BHJ10, E11} and contrary to the results in \cite{GM14, LM15}, both of our results rely on a robust $L^2$ Carleman estimate. While \cite{BHJ10} however relies on Carleman estimates with weights which have strong (pseudo)convexity properties, proving a three balls inequality requires working with (close to) limiting Carleman weights. More precisely, as our key auxiliary result we prove the following Carleman estimate with a weight which is a slightly convexified version of the limiting Carleman weight $\psi(x) = -\tau \log(|x|)$ and which we choose as, for example, in  \cite{KRS16}:

\begin{thm}
\label{thm:Carl}
Let $u: (h\Z)^d \rightarrow \R$ be such that $ h^{-2}\D_{d} u = g$ in $B_4$ with $\supp(u) \subset B_2 \setminus B_{1/2}$ and $g\in L^2(B_4)$. Let $\phi(x):=\tau \varphi(|x|)$,  where
$$
\varphi(t)= -\log t + c_{ps} \Big( \log t \arctan(\log t) - \frac{1}{2} \log(1+\log^2t) \Big)
$$
for a certain small constant $c_{ps}>0$.
Then, there exist $h_0,\delta_0\in(0,1)$ with $h_0<\delta_0$, $C >1$ and $\tau_0>1$ (which are independent of $u$) such that for all $h\in (0,h_0)$ and $\tau \in (\tau_0, \delta_0 h^{-1})$ we have
\begin{equation}
\label{eq:Carl_main}
\tau^3 \|e^{\phi} u\|_{L^2}^2 + \tau \|e^{\phi} h^{-1} D_s u \|_{L^2}^2 + \tau^{-1} \|e^{\phi} h^{-2} D_s^2 u \|_{L^2}^2
\leq C \|e^{\phi} g \|_{L^2}^2.
\end{equation}
Here $D_s u(n) := \frac12\sum\limits_{j=1}^{d} (u(n + h e_j) - u(n-h e_j))$, where $e_j$ is the unit vector in the $j$-th direction, denotes the symmetric discrete difference operator.
\end{thm}

\begin{rmk}
\label{rmk:deriv}
We remark that the choice of the symmetric discrete derivative $D_s$ in \eqref{eq:Carl_main} does not play a substantial role. With only minor changes it is also possible to replace it by $D_+^h$ or $D_-^h$. We refer to the beginning of Section \ref{sec:symmetric} for the precise definitions.

The constraints on the size of the constant $c_{ps}>0$ are specified in the proof of Lemma \ref{lem:weight}.
\end{rmk}

Comparing our estimate with the previous Carleman estimates for discrete operators from \cite{BHJ10, E11}, we emphasize that in proving three spheres inequalities and doubling properties, it is no longer possible to use strongly convex Carleman weights as in \cite{BHJ10}. As a consequence, the derivation of positivity for the commutator becomes more intricate. In this sense, our estimate is closer in spirit to the estimates from \cite{E11}, in which the authors construct discrete complex geometric optics solutions and which thus requires working with \emph{limiting Carleman weights}. However, contrary to \cite{E11} working in a ``unique continuation setting'', we can not rely on ``plane wave'' Carleman weights but have to use the more singular \emph{(almost) logarithmic weights} which we only convexify very slightly. On a technical level this also leads to more complex commutator contributions. Hence, we are confronted with a situation in which only very little pseudoconvexity persists and in which the algebraic, discrete computations become rather involved. In order to overcome this, as one of the main ingredients of our proof, we relate the discrete quantities to their continuous counterparts for which the underlying pseudoconvexity structures become more transparent. 

While building on similar ideas as in its continuous counterpart (see for instance \cite{KT01, AKS62}), our Carleman estimate is restricted to a certain range of values of $\tau$ which is a purely discrete phenomenon. Similar restrictions had earlier been observed in \cite{BHJ10, E11} in the context of Carleman estimates for control theoretic and inverse problems. In deriving this estimate, we localize to suitable scales on which we freeze coefficients and compare our discrete problem to the continuum setting.

\subsection{Outline of the article}
The remainder of the article is organized as follows: In Section \ref{sec:symmetric} we compute the conjugated discrete operator and its expansion into its symmetric, antisymmetric parts and their commutator. In the main part, in Section \ref{sec:Carl}, we derive the main Carleman estimate of Theorem \ref{thm:Carl}. Building on this, in Section \ref{sec:proofs_main} we deduce the results of Theorems \ref{thm:log_conv} and \ref{thm:3spheres}. Last but not least, in Section \ref{sec:scaling}, we comment on rescaled versions of the main estimates.

\subsection{Remarks on the notational conventions}
Concerning notation, with the letters $c, C,\ldots$ we denote structural constants that depend only on the dimension and on parameters that are not relevant. Their values might vary from one occurrence to another, and in most of the cases we will not track the explicit dependence. For the Fourier transform of a function $f$ we will use the notation $\widehat{f}$.

\section{The Conjugated Laplacian and the Commutator}
\label{sec:symmetric}

From now on, $D_{\pm}^j$  will stand for the forward/backward operators $D_{\pm,j}^h $ from Section \ref{sec:intro} and $ D_{s}^j$ will denote the symmetric discrete derivatives in the $j$-th direction, i.e. $D_s^j u(n):= \frac{1}{2}(u(n+h e_j)- u(x- h e_j))$. All operators are understood to be taken with step size $h$. Moreover, $D_{\pm}^h:=  \sum\limits_{j=1}^d D_{\pm}^j $ and $D_{s}:=  \sum\limits_{j=1}^d D_{s}^j$. We remark that the symmetric difference operator is associated with the Fourier multiplier $i\sum\limits_{j=1}^{d} \sin(h \xi_j)$, where $i$ denotes the complex unit.

Heading towards the proof of the Carleman inequality of Theorem \ref{thm:Carl}, we introduce the conjugated Laplacian
\begin{equation}
\label{SAphi}
L_{\phi,h} f(n):= h^{-2}e^{\phi} \D_d e^{-\phi} f(n) = \sum\limits_{j=1}^{d} \big[ S_j f(n) + A_j f(n) \big]=:(S_{\phi}f(n)+A_{\phi}f(n)),
\end{equation}
where the symmetric and anti-symmetric operators are
\begin{align*}
S_jf(n) &=h^{-2}((\cosh(D_+^j\phi(n)))f(n+ h e_j)+ (\cosh(D_{-}^j\phi(n)))f(n- h e_j)-2f(n)),\\
A_jf(n) &=h^{-2}((\sinh(D_{-}^j\phi(n)))f(n- h e_j)- (\sinh(D_+^j\phi(n)))f(n+he_j)).
\end{align*}
We compute the commutator of this to be
\begin{align*}
[S,A]f(n) &=  \sum_{j,k}[S_j,A_k]f(n)=\sum_{j,k}\{S_jA_k-A_k S_j\}f(n)
\\
&=\frac12\sum_{j,k}\{S_jA_k-A_k S_j\}f(n) +\frac12\sum_{j,k}\{S_kA_j-A_j S_k\}f(n)\\
&=:\sum_{j,k}h^{-4}T_{j,k}f(n),
\end{align*}
where
\begin{align*}
T_{j,k} f(n)&:= h^4[S_j, A_k]f(n) =A_{j,k} f(n+ h e_j+ h e_k)\\
&\quad+B_{j,k} f(n- h e_j- h e_k)+ C_{j,k}f(n+ h e_j- h e_k) + E_{j,k}f(n- h e_j+ h e_k),
\end{align*}
with
\begin{align*}
A_{j,k} &= -\cosh(D^j_+ \phi(n)) \sinh(D_+^k \phi(n+ h e_j)) + \sinh(D_+^k \phi(n))\cosh(D^j_+ \phi(n+ h e_k)),\\
B_{j,k} & = \cosh(D^j_-\phi(n))\sinh(D^k_- \phi(n-h e_j)) - \sinh(D^k_- \phi(n))\cosh(D^j_-\phi(n- h e_k)),\\
C_{j,k} & = \cosh(D^j_+\phi(n))\sinh(D^k_- \phi(n+ h e_j))-\sinh(D^k_- \phi(n))\cosh(D^j_+ \phi(n- h e_k)),\\
E_{j,k} & = -\cosh(D^j_- \phi(n))\sinh(D^k_+\phi(n- h e_j)) + \sinh(D^k_+ \phi(n)) \cosh(D^j_-\phi(n+ h e_k)).
\end{align*}
Now, using trigonometric identities, these can be simplified to read
\begin{align*}
A_{j,k}&=-\sinh(D_+^jD_+^k\phi(n))\cosh(D_+^j\phi(n)-D_+^k\phi(n)),\\
B_{j,k}&=-\sinh(D_-^jD_-^k\phi(n))\cosh(D_-^j\phi(n)-D_-^k\phi(n)),\\
C_{j,k}&=\sinh(D_+^jD_-^k\phi(n))\cosh(D_+^j\phi(n)+D_-^k\phi(n)),\\
E_{j,k}&=\sinh(D_-^jD_+^k\phi(n))\cosh(D_-^j\phi(n)+D_+^k\phi(n)).
\end{align*}
Indeed, for instance, for $A_{j,k}$ we obtain
\begin{align*}
A_{j,k}&=-\cosh(D^j_+ \phi(n)) \sinh(D_+^k \phi(n+ h e_j)) + \sinh(D_+^k \phi(n))\cosh(D^j_+ \phi(n+ h e_k))\\
&= - \cosh(D^j_+ \phi(n)) \sinh(D^k_+ D^j_+ \phi(n) + D^k_+ \phi(n))\\
& \quad + \cosh(D^j_+ D^k_+ \phi(n) + D^j_+ \phi(n)) \sinh(D^k_+ \phi(n))\\
& = -\cosh(D^j_+ \phi(n))\big[ \sinh(D^k_+ D^j_+ \phi(n)) \cosh(D^k_+ \phi(n)) + \sinh(D^k_+ \phi(n))\cosh(D^k_+ D^j_+ \phi(n)) \big]\\
&\quad  +  \sinh(D^k_+\phi(n))\big[ \cosh(D^j_+ D^k_+ \phi(n))\cosh(D^j_+ \phi(n)) + \sinh(D^j_+ D^k_+ \phi(n)) \sinh(D^j_+ \phi(n)) \big]\\
& = \sinh(D^j_+ D^k_+ \phi(n))\big( \sinh(D^j_+ \phi(n))\sinh(D^k_+\phi(n)) - \cosh(D^k_+\phi(n))\cosh(D^j_+\phi(n)) \big)\\
& = - \sinh(D^j_+ D^k_+ \phi(n)) \cosh(D^j_+\phi(n) - D^k_+ \phi(n)).
\end{align*}
The arguments for the other contributions are similar.

We next seek to investigate the commutator in more detail.

\begin{rmk}
\label{rmk:oneD}
In the one-dimensional situation the commutator can be simplified significantly: Indeed, if we study the commutator term $\langle [S,A]f,f\rangle$, the case $j=k=1$ is quite simple and leads to
\[
\sum_{ n\in  \mathbb{Z} }\Big\{4\sinh(\Delta_d \phi(n))|D_s f(n)|^2-\Delta_d\sinh(\Delta_d \phi(n))|f(n)|^2+2\sinh(\Delta_d \phi(n))(\cosh(2D_s\phi(n))-1)|f(n)|^2\Big\}.
\]
The main term is a discrete version of $4\phi_{jj}|f_j|^2+4\phi_{jj} \phi_j\phi_j|f|^2- \phi_{jjjj}|f|^2$, where the subindices refer to differentiation in the corresponding direction. Note that the main term of the higher dimensional continuous commutator is more complicated and is of the form
\[
 4\phi_{jk}f_j \overline{f_k}+4\phi_{jk} \phi_j\phi_k|f|^2- \phi_{jjkk}|f|^2.
\]
\end{rmk}

In the general case, we can rewrite the contributions of $h^4 \langle [S,A]f,f\rangle$ in the following way (where with slight abuse of notation, we refrain from spelling out the sums in $(h\mathbb{Z})^d$ and the sum in $j,k$):
\begin{align}
\label{eq:comm_spell_out}
\begin{split}
&\sinh(D_{++}^{j,k}\phi(n))f(n+ h e_j)\overline{f(n+ h e_k)}+\sinh(D_{--}^{j,k}\phi(n))f(n- h e_j)\overline{f(n- h e_k)}\\
&\quad -\sinh(D_{+-}^{j,k}\phi(n))f(n+ h e_j)\overline{f(n- he_k)}-\sinh(D_{-+}^{j,k}\phi(n))f(n- h e_j)\overline{f(n+ h e_k)}\\
&\quad+\sinh(D_{++}^{j,k}\phi(n))\big(\cosh(\phi(n+ h e_j+ h e_k)-\phi(n))-1\big)f(n+ h e_j)\overline{f(n+ h e_k)}\\
&\quad+\sinh(D_{--}^{j,k}\phi(n))\big(\cosh(\phi(n- h e_j- h e_k)-\phi(n))-1\big)f(n- h e_j)\overline{f(n- h e_k)}\\
&\quad-\sinh(D_{+-}^{j,k}\phi(n))\big(\cosh(\phi(n+h e_j- h e_k)-\phi(n))-1\big)f(n+ h e_j)\overline{f(n- h e_k)}\\
&\quad-\sinh(D_{-+}^{j,k}\phi(n))\big(\cosh(\phi(n- h e_j+ h e_k)-\phi(n))-1\big)f(n- h e_j)\overline{f(n+ h e_k)}.
\end{split}
\end{align}

The interest of writing the general term in this form is that we seek to bring the commutator term into a form which is as close as possible to the form of the commutator in the continuous setting which reads
\begin{equation*}
4( \nabla \phi \cdot \nabla^2 \phi \nabla \phi) f^2 + 4 \nabla f \cdot \nabla^2 \phi \nabla f - \D^2 \phi  f^2.
\end{equation*}
To this end, we note that the first four terms in \eqref{eq:comm_spell_out} are closely related to the part  $4\phi_{jk}f_j \overline{f_k}- \phi_{jjkk}|f|^2$ and the last four terms to $4\phi_{jk} \phi_j\phi_k|f|^2$ correspondingly.

We will use the expression \eqref{eq:comm_spell_out} as the starting point of our commutator estimates in the following sections.

\section{Proof of the Carleman Estimate from Theorem \ref{thm:Carl}}

\label{sec:Carl}

Before turning to the proof of Theorem \ref{thm:Carl} let us recall an auxiliary result showing the strong pseudoconvexity (in the continuous sense) of the weight function $\phi(x)$:

\begin{lem}
\label{lem:weight}
Let $\phi(x):=\tau \varphi(|x|)$, where for some small constant $c_{ps}>0$
\begin{equation}
\label{eq:weight_convexified_1}
\varphi(t)= -\log t + c_{ps} \Big( (\log t) \arctan(\log t) - \frac{1}{2} \log(1+(\log t)^2) \Big).
\end{equation}
Then $\phi(x)$ is strongly pseudoconvex with respect to the Laplacian and with respect to the domain $B_4 \setminus B_1$ in the sense that there exists a constant $C>0$ (which is independent of $\tau$) such that in $B_4 \setminus B_1$ we have
\begin{equation*}
\nabla \phi \cdot \nabla^2 \phi \nabla \phi + \xi \cdot \nabla^2 \phi \, \xi \geq C\tau^3 >0 \mbox{ on } \{|\xi|^2 = |\nabla \phi|^2, \ \nabla \phi \cdot \xi = 0\}.
\end{equation*}
\end{lem}

\begin{proof}
In order to prove pseudoconvexity, we seek to prove that
\begin{equation*}
\nabla \phi \cdot \nabla^2 \phi \nabla \phi + \xi \cdot \nabla^2 \phi \,\xi\ge C \tau^3 >0 \mbox{ on } \{|\xi|^2 = |\nabla \phi|^2 ,\ \nabla \phi \cdot \xi = 0\}.
\end{equation*}
Without loss of generality, we consider the case $\tau=1$ only; the general case follows then by rescaling. Now, if $\phi(x) = \varphi(|x|)$ we have that
\begin{equation*}
\nabla^2 \phi(x) = \frac{\varphi'(|x|)}{|x|} \operatorname{Id} + \Big( \varphi''(|x|)-\frac{\varphi'(|x|)}{|x|}\Big)\frac{x}{|x|}\otimes \frac{x}{|x|}.
\end{equation*}
Moreover, $\nabla \phi(x)$ is an eigenvector of $\nabla^2 \phi(x)$ with eigenvalue $\lambda(x)=  \varphi''(|x|)$ and
\begin{equation*}
\nabla \phi(x) \cdot \nabla^2 \phi(x) \nabla \phi(x) = \varphi''(|x|)(\varphi'(|x|))^2.
\end{equation*}
Furthermore, the only other eigenvalue (of multiplicity $d-1$) of $\nabla^2 \varphi(|x|)$ is given by $\mu(|x|) = \frac{\varphi'(|x|)}{|x|}$. Due to the constraint that $ \{|\xi|^2 = |\nabla \phi|^2, \ \nabla \phi \cdot \xi = 0\}$, we therefore infer that
\begin{equation*}
\nabla \phi(x) \cdot \nabla^2 \phi(x) \nabla \phi(x) + \xi \cdot \nabla^2 \phi(x) \xi
= (\varphi'(|x|))^2 \Big(\varphi''(|x|)+\frac{\varphi'(|x|)}{|x|} \Big).
\end{equation*}
For $\phi(x)= -\tau \log(|x|)$ this vanishes (as it is a limiting Carleman weight) but for \eqref{eq:weight_convexified_1} one obtains that on the critical set
\begin{equation*}
\nabla \phi \cdot \nabla^2 \phi \nabla \phi + \xi \cdot \nabla^2 \phi \,\xi
= (\varphi'(|x|))^2 \Big(\varphi''(|x|)+\frac{\varphi'(|x|)}{|x|} \Big)
= c_{ps}\frac{(-1 + c_{ps} \arctan(\log(|x|)))^2}{|x|^4(1+\log^2(|x|))},
\end{equation*}
which is positive for some sufficiently small $c_{ps}>0$. 
\end{proof}

In the sequel, we present several auxiliary results which allow us to steadily transform the discrete conjugated operator into an operator that closely resembles the continuum version of the conjugated Laplacian. Recall that we define the discrete Laplacian in direction $j\in\{1,\dots, d\}$ as
\begin{equation*}
\D_{d,h,j}f(n) = f(n + h e_j) + f(n - h e_j) - 2 f(n).
\end{equation*}
As a first step towards the desired Carleman estimate, we localize the problem to scales of order $\epsilon_0^{-1} \tau^{-\frac{1}{2}}$, where $\epsilon_0>0$ is a small constant which will be chosen below (see the proof of Theorem \ref{thm:Carl}):

 \begin{lem}
\label{lem:localize}
Let $\phi$ be as in Lemma \ref{lem:weight}. Let $S_{\phi}$, $A_{\phi}$ be as in \eqref{SAphi}. Let $\{\psi_k(x)\}_{k\in \Z}$ be a partition of unity subordinate to an open cover of $B_2 \setminus B_{\frac{1}{2}}$ which is localized to scales of order $\epsilon_0^{-1} \tau^{-\frac{1}{2}}$ for some $\epsilon_0 \in (0,1)$ small. Suppose further that $\tau \in (1,\delta_0 h^{-1})$, where $h \in (0, h_0)$ for $\delta_0, h_0 \in (0,1)$ is chosen to be small and such that $h_0<\delta_0$. Let $f_k(n):= (f\psi_k)(n)$. Then,
\begin{align}
\label{eq:localize}
\begin{split}
\|S_{\phi} f\| &\leq \sum\limits_{k} \|S_{\phi} f_k\| \leq C \|S_{\phi} f\| + C \tau^{\frac{1}{2}}\epsilon_0 \sum\limits_{j=1}^d \|h^{-1}D_s^jf\| + C (\tau \epsilon_0 + \tau^2 \tau^{\frac{1}{2}} h \epsilon_0)  \|f\|,\\
\|A_{\phi} f\| &\leq \sum\limits_{k} \|A_{\phi} f_k\| \leq C \|A_{\phi} f\| + C \tau^{\frac32} \epsilon_0 \|f\|,\\
\|L_{\phi} f\| &\leq \sum\limits_{k} \|L_{\phi} f_k\| \leq C \|L_{\phi} f\| + C \tau^{\frac{1}{2}}\epsilon_0 \sum\limits_{j=1}^d \|h^{-1}D_s^jf\| + C (\tau \epsilon_0 +\tau^{\frac{3}{2}} \epsilon_0 + \tau^2 \tau^{\frac{1}{2}} h \epsilon_0)  \|f\|.
\end{split}
\end{align}
\end{lem}

We remark that the condition $h_0<\delta_0$ is imposed in order to ensure that for all $ h\in(0,h_0)$, we have $\delta_0h^{-1}>1$. Here and in the sequel, for brevity of notation, we write the $L^2$ norm on the lattice without adding subindices, i.e. $\|f\| := \|f\|_{L^2((h\Z))^d}$.

\begin{proof}[Proof of Lemma \ref{lem:localize}]
As the estimates for $S_{\phi}$ and for $A_{\phi}$ are analogous, we mainly focus on the argument for $S_{\phi}$. The first bound in the estimate for $S_{\phi}$ in \eqref{eq:localize} is a direct consequence of Minkowski's inequality. In order to observe the second estimate for $S_{\phi}$ in \eqref{eq:localize}, we spell out the contributions coming from $S_{\phi} f_k(n)$. We begin by rewriting
\begin{multline*}
S_j f_k(n)  \\
= h^{-2}(\cosh(D^j_+ \phi(n))-1) f_k(n + h e_j)
+ h^{-2}(\cosh(D^j_- \phi(n))-1) f_k(n - h e_j) + h^{-2}   \D_{d,h,j} f_k(n).
\end{multline*}
Hence, inserting the function $f_k(n)$ for $h^{-2}  \D_{d,h,j} f_k(n)$, we obtain
\begin{align*}
h^{-2} \D_{d,h,j}f_k(n)
& = h^{-2}[f(n+h e_j) + f(n- h e_j) - 2 f(n)] \psi_k(n)\\
& \quad + h^{-2}[f(n+ h e_j)(\psi_k(n + h e_j)-\psi_k(n)) + f(n-h e_j)(\psi_k(n-h e_j)-\psi_k(n))]\\
& = \psi_k(n) (h^{-2}  \D_{d,h,j} f(n))
+ h^{-2}(f(n + h e_j) - f(n-h e_j))(\psi_k(n + h e_j)-\psi_k(n-h e_j))\\
& \quad + h^{-2} (f(n+h e_j)-f(n-h e_j))(\psi_k(n-h e_j)-\psi_k(n))\\
& \quad 
+ f(n-h e_j)h^{-2}( \D_{d,h,j}\psi_k)(n).
\end{align*}
While we seek to keep the first contribution in this expansion to recombine it to $h^{-2}  \D_{d,h,j} f$ after summing over the partition of unity, we only provide estimates on the remaining contributions. To this end,  denoting by $2h\operatorname{supp}( \psi_k )$  a $2h$-neighbourhood of the support of $\psi_k$, we observe that
\begin{align*}
&|h^{-2}(f(n + h e_j) - f(n-h e_j))(\psi_k(n + h e_j)-\psi_k(n-he_j))|\\
&\leq C|h^{-1}(f(n + h e_j) - f(n-h e_j))| |\nabla \psi_k|
\leq
C \epsilon_0 \tau^{\frac{1}{2}} |h^{-1}(f(n + h e_j) - f(n-h e_j))|  \chi_{2h\operatorname{supp}( \psi_k )}(n),
\end{align*}
as
$$
|\nabla \psi_k(n)| \le C \epsilon_0 \tau^{\frac{1}{2}} \chi_{2h\operatorname{supp}( \psi_k )}(n),\quad \text{ for all } k
 $$
and
\begin{equation*}
|f(n-h e_j)h^{-2}( \D_{d,h,j}\psi_k)(n)| \leq C|f(n-h e_j)||D^2 \psi_k|
\leq C \tau \epsilon_0^2|f(n-h e_j)| \chi_{2h\operatorname{supp}( \psi_k )}(n).
\end{equation*}
Using the same reasoning for the term $h^{-2} (f(n+h e_j)-f(n-h e_j))(\psi_k(n-h e_j)-\psi_k(n))$, and combining these estimates, we thus infer that
\begin{align}
\label{eq:part1}
\begin{split}
& |h^{-2}  \D_{d,h,j}f_k(n)  - \psi_k(n) h^{-2}  \D_{d,h,j}f(n)|\leq C\tau \epsilon_0^2|f(n-h e_j)| \chi_{2h\operatorname{supp}( \psi_k )}(n)  \\
&+ C \epsilon_0 \tau^{\frac{1}{2}} |h^{-1}(f(n + h e_j) - f(n-h e_j))| \chi_{2h\operatorname{supp}( \psi_k )}(n).
\end{split}
\end{align}
Similarly, for $h^{-2} (\cosh(D^j_+ \phi(n))-1) f_k(n + h e_j)$ we obtain
\begin{multline*}
h^{-2} (\cosh(D^j_+ \phi(n))-1) f_k(n + h e_j)
= \psi_k(n) h^{-2} (\cosh(D^j_+ \phi(n))-1) f(n + h e_j)\\
 + h^{-2} (\cosh(D^j_+ \phi(n))-1) f(n + h e_j)(\psi_k(n+ h e_j)- \psi_k(n)).
\end{multline*}
Estimating
\begin{equation*}
|\cosh(D^j_+ \phi(n))-1|
\leq \frac{|D^j_+ \phi(y)|^2}{2}
\leq \tau^2 h^2 \frac{|\nabla \varphi(\widetilde{y})|^2}{2},
\end{equation*}
where we used that for $y \in B_{2}\setminus B_{\frac{1}{2}}$ we have that 
\begin{equation*}
|D^j_+\phi(n)| = h |\partial_j \phi(y)| = h \tau |\partial_j \varphi(y)| \leq C \delta_0 
\end{equation*}
is small, allowing for a Taylor expansion of the hyperbolic cosine.  
Here $\widetilde{y}$, $y \in [n, n+ h e_j]$ are intermediate values. We hence obtain
\begin{align}
\label{eq:err_cosh}
\begin{split}
& C h^{-1} |(\cosh(D^j_+ \phi(n))-1) f(n + h e_j) (\psi_k(n + he_j) - \psi_k(n))| h^{-1}\\
&\leq C \tau^2 \tau^{\frac{1}{2}} h \epsilon_0|f(n+he_j)|  \chi_{2h\operatorname{supp}( \psi_k )}(n) .
\end{split}
\end{align}

As a consequence, combining the estimates from \eqref{eq:part1} and \eqref{eq:err_cosh} yields
\begin{align*}
\sum\limits_k \| S_{\phi} f_k\|
&\leq C \sum\limits_{k} \| \psi_k S_{\phi} f\|
+ \sum\limits_{k}\Big( C \tau^{\frac{1}{2}}\epsilon_0 \sum\limits_{j=1}^d \|\psi_k h^{-1}D_s^jf\| + C( \tau \epsilon_0   + \tau^2 \tau^{\frac{1}{2}} h \epsilon_0  )  \| f \chi_{2h\operatorname{supp}( \psi_k )}\|  \Big)\\
& \leq C \|S_{\phi} f \| + C \tau^{\frac{1}{2}}\epsilon_0 \sum\limits_{j=1}^d \| h^{-1}D_s^jf\| + C( \tau \epsilon_0 + \tau^2 \tau^{\frac{1}{2}} h \epsilon_0 ) \| f\| .
\end{align*}
This concludes the argument for the localization estimate for $S_{\phi}$.

The arguments for $A_{\phi}$ and $L_{\phi}$ are analogous.
Indeed, for $A_{\phi}$ we note that, for an intermediate value $y$,
\begin{multline*}
|A_j(\psi_k f)(n)-\psi_k A_jf (n)|\le
|-h^{-1}\partial_j\phi(y) f(n+he_j)(\psi_k(n+he_j)-\psi_k(n))\\+h^{-1}\partial_j\phi(y)f(n-he_j)(\psi_k(n-he_j)-\psi_k(n))|,
\end{multline*}
which yields
$$
 \sum\limits_{k} \|A_{\phi} f_k\| \leq C \|A_{\phi} f\| + C \tau^{3/2} \epsilon_0 \|f\|.
$$
Estimating the terms of $L_{\phi}$ by using the bounds for $A_{\phi}$ and $S_{\phi}$ then implies the result.
\end{proof}

As a next auxiliary step, we expand the trigonometric identities which then allows for easier manipulations of the contributions in the sequel.

\begin{lem}
\label{lem:discrete_approx}
Let $\phi$ be as in Lemma \ref{lem:weight}.
Let
\begin{align*}
\begin{split}
S_j f(n) &= h^{-2}\big(\cosh(D^j_+ \phi(n)) f(n+ h e_j) + \cosh(D_-^j\phi(n))f(n- h e_j) - 2f(n) \big),\\
A_j f(n) & = -h^{-2}\sinh(D^j_+ \phi(n)) f(n + h e_j) + h^{-2}\sinh(D^j_- \phi(n)) f(n- h e_j),
\end{split}
\end{align*}
and $[S_j, A_k] f(n)$ be the quantities from Section \ref{sec:symmetric}. Let further
\begin{align*}
\begin{split}
\widetilde{S}_j f(n) & := h^{-2}\Big(\Big(\frac{h^2 (\p_j \phi(n))^2}{2}+1 \Big) f(n+ h e_j) + \Big(\frac{h^2 (\p_j \phi(n))^2}{2} +1 \Big) f(n-h e_j) - 2 f(n) \Big)\\
& = h^{-2}\D_{d,h,j} f(n) + \frac{(\p_j \phi(n))^2}{2}\big( f(n + h e_j) + f(n - h e_j) \big),\\
\widetilde{A}_j f(n) & :=
-h^{-1}(\p_j \phi(n))( f(n+ h e_j) - f(n-h e_j)),\\
\mathcal{C}_{jk}^{f,f}(n) & := (\p_{jk}\phi(n))\left(h^{-1}(f(n+he_j)-f(n-h e_j))\overline{h^{-1}(f(n+h e_k)-f(n- he_k))} \right)\\
& \qquad + \frac12\p_{jk}\phi(n)\big(|\p_j \phi(n)+\p_k \phi(n)|^2(f(n+he_j)\overline{f(n+h e_k)} + f(n-h e_j) \overline{f(n-he_k)}) \\
&\qquad  - |\p_j \phi(n) - \p_k \phi(n)|^2(f(n+he_j)\overline{f(n-he_k)} + f(n-he_j)\overline{f(n+h e_k)}) \big).
\end{split}
\end{align*}
Let $\tau \in (1, h^{-1} \delta_0)$, where $h\in(0,h_0)$ with $\delta_0 \in (0,1)$ (to be chosen below, see the proof of Theorem \ref{thm:Carl}) and $h_0<\delta_0$.

Then, for $S_{\phi}$ and $A_{\phi}$ as in \eqref{SAphi}, $\widetilde{S}_{\phi} f(n):= \sum\limits_{j=1}^{d} \widetilde{S}_{j} f(n)$ and $\widetilde{A}_{\phi} f(n) := \sum\limits_{j=1}^{d} \widetilde{A}_{j} f(n)$, and $f\in L^2(B_2 \setminus B_{\frac{1}{2}})$ with $\supp(f) \subset B_2 \setminus B_{\frac{1}{2}}$, we have
\begin{align*}
\|S_{\phi} f\|^2 &\geq \|\widetilde{S}_{\phi} f\|^2 - 2 C(\delta_0^2 \tau^2 + \delta_0^4 \tau^4) \| f\|^2,\\
 \|A_{\phi} f\|^2&  \geq \|\widetilde{A}_{\phi} f\|^2 - 2 C(\tau^2+\delta_0^2\tau^4) \|f\|^2,\\
([S_{\phi}, A_{\phi}]f, f) &\geq \sum_{n\in(h\Z)^d}\sum\limits_{j,k=1}^{d}\mathcal{C}_{jk}^{f,f}(n) - 2C(\tau^2+\delta_0^2\tau^3) \| f \|^2 - 2C \sum\limits_{j=1}^{d} \|h^{-1}D_s^j f\|^2.
\end{align*}
\end{lem}

Similarly as above, we drop the subscript in the $L^2$ scalar product and simply write $(\cdot, \cdot ) := (\cdot, \cdot)_{L^2((h \Z)^d)}$.

\begin{proof}[Proof of Lemma \ref{lem:discrete_approx}]
The results follow by expanding the expressions for $S_j, A_j$. More precisely, we first approximate all discrete derivatives of $\phi$ and the corresponding nonlinear functions and then estimate the resulting errors.

\emph{Step 1: The symmetric part.}
We first discuss the symmetric part of the operator. For instance, we expand 
\begin{align*}
\begin{split}
\cosh(D^j_+ \phi(n))
&= \cosh(h\p_j \phi(n) + O(h^2|\nabla^2 \phi(y)|))\\
&= 1 + \frac{1}{2}\big|h \p_j \phi(n) + O(h^2|\nabla^2 \phi(y)|)\big|^2 + O\big((|h \nabla \phi(y)|+ |h^2 \nabla^2 \phi(y)|)^4\big)\\
&= 1 + \frac{1}{2}h^2|\p_j \phi(n)|^2 + O(h^3(| \nabla^2 \phi(y)|^2+|\nabla \phi(y)|^2) + h^4|\nabla \phi(y)|^4 + h^8 |\nabla^2 \phi(y)|^{4}).
\end{split}
\end{align*}
Here $y \in \R^d$ are intermediate values, not necessarily the same, such that $y \in [n,n+he_j]$.
Thus, the symmetric part becomes
$$
S_j f(n) =\widetilde{S}_{j} f(n) - E_{S_{j}}f(n),
$$
where $\widetilde{S}_j f(n)$ is as in our statement
and
$$
\|E_{S_{j}} f\| \leq C(h \tau^2 +   \tau^4h^2)\| (|\nabla \varphi|^2 + |\nabla \varphi|^4+|\nabla^2 \varphi|^2 +|\nabla^2 \varphi|^4) f\|,
$$
with $n \in (h \Z)^d$, $\phi(n)=\tau \varphi(n)$ with $\varphi$ a bounded function (on the relevant domain). Choosing $\tau \in (1, \delta_0 h^{-1})$ with $\delta_0$ sufficiently small, we may assume that $h \tau^2 + \tau^4 h^2 \leq C(\delta_0 \tau + \delta_0^2 \tau^2 )$, hence the error $\|E_{S_{j}} f\|$ in the symmetric part is an $L^2$ contribution and, combining this with the explicit form of $\varphi$, satisfies the estimate $\|E_{S_{j}} f\| \leq C \delta_0 (\tau + \delta_0 \tau^2) \|f\|$. Therefore, in the sequel, we will estimate
$$
\|S_{\phi}f\|^2 \geq \|\widetilde{S}_{\phi} f\|^2 - 2 \sum\limits_{j=1}^{d}\|E_{S_j} f\|^2.
$$

\emph{Step 2: The antisymmetric part.}
For the antisymmetric part we argue analogously. We thus expand
\begin{align*}
\sinh(D^j_+ \phi(n))
&= \sinh(h \p_j \phi(n) + O(h^2 |\nabla^2 \phi(y)|))\\
& = h \p_j \phi(n) + O(h^2 |\nabla^2 \phi(y)|)+O\big((h |\nabla \phi(y)| + O(h^2 |\nabla^2 \phi(y)|))^3\big)\\
& = h \p_j \phi(n) +O(h^2 |\nabla^2 \phi(y)|+h^3|\nabla\phi(y)|^3+h^6|\nabla^2\phi(y)|^3),
\end{align*}
where $y$ are intermediate values in $[n,n+he_j]$.
Thus, the antisymmetric part becomes
$$
A_j f(n) = \widetilde{A}_j f(n) - E_{A_j} f(n)
$$
with
\begin{align*}
\widetilde{A}_{j}f(n) &=-h^{-1}(\p_j \phi(n)) (f(n+ h e_j) - f(n-h e_j)),\\
 \| E_{A_j}f\| &\leq C (\tau+\delta_0\tau^2) \|(|\nabla^2 \varphi|+ |\nabla\varphi|^3+|\nabla^2\varphi|^3)f\|\\
 &\leq C (\tau+\delta_0\tau^2) \|f\|,
\end{align*}
for which we have used the bounds for $\varphi $ in $B_2 \setminus B_{\frac{1}{2}}$.

\emph{Step 3: The commutator.}
Finally, we turn to the commutator which is given by
\begin{align}
\label{eq:comm1a}
\begin{split}
&\sum_{n\in (h\Z)^d}\sum_{j,k=1}^df(n)\overline{[S_j, A_k] f(n)}
\\
&= \sum_{n\in (h\Z)^d}\sum_{j,k=1}^dh^{-4}\Big(
\sinh(D_{++}^{j,k}\phi(n))f(n+ he_j)\overline{f(n+ h e_k)}+\sinh(D_{--}^{j,k}\phi(n))f(n- h e_j)\overline{f(n-h e_k)}  \\
&\quad-\sinh(D_{+-}^{j,k}\phi(n))f(n+ h e_j)\overline{f(n-h e_k)}-\sinh(D_{-+}^{j,k}\phi(n))f(n-h e_j)\overline{f(n+ h e_k)} \\
&\quad+\sinh(D_{++}^{j,k}\phi(n))\big(\cosh(\phi(n+ he_j+ h e_k)-\phi(n))-1\big)f(n+ he_j)\overline{f(n+ h e_k)}\\
&\quad+\sinh(D_{--}^{j,k}\phi(n))\big(\cosh(\phi(n-h e_j- h e_k)-\phi(n))-1\big)f(n- h e_j)\overline{f(n- h e_k)}\\
&\quad-\sinh(D_{+-}^{j,k}\phi(n))\big(\cosh(\phi(n+h e_j- h e_k)-\phi(n))-1\big)f(n+ h e_j)\overline{f(n- h e_k)}\\
&\quad-\sinh(D_{-+}^{j,k}\phi(n))\big(\cosh(\phi(n- h e_j+ h e_k)-\phi(n))-1\big)f(n- h e_j)\overline{f(n+ h e_k)} \Big).
\end{split}
\end{align}
For the first four contributions in \eqref{eq:comm1a}, we expand, for each $n\in (h\Z)^d$ and fixed $j,k\in\{1,\dots,d\}$,
\begin{align}
\begin{split}
\label{sinh}
\sinh(D^{j,k}_{++}\phi(n)) &= h^2 \p_{jk}\phi(n) + \frac12 h^3(\p_k \p_j^2  +  \p_{k}^2 \p_j)\phi(n) \\
& \quad + O(h^4 |\nabla^4 \phi(y)|  +  h^6\max\{|\nabla^2 \phi(y)|, |\nabla^3 \phi(y)|, |\nabla^4 \phi(y)|\}^3),\\
\sinh(D^{j,k}_{--}\phi(n)) &= h^2 \p_{jk}\phi(n) -  \frac12 h^3(\p_k \p_j^2 + \p_k^2 \p_j)\phi(n) \\
& \quad + O(h^4 |\nabla^4 \phi(y)|  +  h^6\max\{|\nabla^2 \phi(y)|, |\nabla^3 \phi(y)|, |\nabla^4 \phi(y)|\}^3),\\
\sinh(D^{j,k}_{+-}\phi(n)) &= h^2 \p_{jk}\phi(n) - \frac12 h^3(\p_j \p_k^2 - \p_j^2 \p_k)\phi(n)\\
& \quad + O(h^4 |\nabla^4 \phi(y)|  +  h^6\max\{|\nabla^2 \phi(y)|, |\nabla^3 \phi(y)|, |\nabla^4 \phi(y)|\}^3)\\
\sinh(D^{j,k}_{-+}\phi(n)) &= h^2 \p_{jk}\phi(n) + \frac12 h^3(\p_j \p_k^2 - \p_j^2 \p_k)\phi(n) \\
& \quad + O(h^4 |\nabla^4 \phi(y)|  +  h^6\max\{|\nabla^2 \phi(y)|, |\nabla^3 \phi(y)|, |\nabla^4 \phi(y)|\}^3)
\end{split}
\end{align}
with $y$ intermediate points. Here we have carried out Taylor expansions of both the functions $D^{j,k}_{\pm \pm}\phi(n)$ and of $\sinh(\cdot)$. 
Thus, the first four contributions in \eqref{eq:comm1a} can be written as
\begin{align*}
\begin{split}
& h^{-2} (\p_{jk}\phi(n)) \big(f(n+he_j)\overline{f(n+h e_k)} + f(n-he_j)\overline{f(n-h e_k)}\\
& \qquad -f(n+he_j)\overline{f(n-h e_k)}-f(n-h e_j)\overline{f(n+h e_k)} \big) + E_1(n,j,k) + E_2(n,j,k),
\end{split}
\end{align*}
where 

\begin{align*}
E_1(n,j,k) &:=  \frac12\p_k \p_j^2 \phi(n) f(n+ h e_j)[h^{-1}\overline{(f(n+he_k)-f(n-he_k))}]\\
&\quad +  \frac12 \p_k^2 \p_j \phi(n)\overline{f(n + h e_k)}[h^{-1}(f(n+ h e_j)-f(n-he_j))]\\
& \quad +  \frac12 \p_k^2 \p_j \phi(n) [h^{-1}(f(n+he_j)-f(n-he_j)) \overline{f(n-he_k)}]\\
& \quad +  \frac12\p_k \p_j^2 \phi(n) [h^{-1}\overline{(f(n+he_k)-f(n-he_k)))}f(n-h e_j)],
\end{align*}

and 
\begin{align*}
|E_2(n,j,k)| &\leq C (\tau\max\{|\nabla^4 \varphi(y)|,\delta_0^2|\nabla^2\varphi(y)|^3, \delta_0^2 |\nabla^3 \varphi(y)|^3,\delta_0^2 |\nabla^4 \varphi(y)|^3)\})\times\\
& \quad \times (|f(n+he_j)|^2 + |f(n-he_j)|^2 + |f(n + h e_k)|^2 + |f(n- he_k)|^2).
\end{align*}

Note that
\begin{align*}
&h^{-2}\big(f(n+he_j)\overline{f(n+h e_k)} + f(n-he_j)\overline{f(n-h e_k)} \\
& \qquad -f(n+he_j)\overline{f(n-h e_k)}-f(n-h e_j)\overline{f(n+h e_k)} \big)\\
&\quad =
\big(h^{-1}(f(n+he_j)-f(n-h e_j))\overline{h^{-1}(f(n+h e_k)-f(n- he_k))} \big),
\end{align*}
which yields the first part in the expression which is claimed for $\mathcal{C}_{jk}^{f,f}$ in the lemma. Moreover, the error $E_1$ can be bounded by the Cauchy-Schwarz inequality:
\begin{align*}
\sum_{j,k=1}^{d}\sum_{n \in (h\Z)^d} | E_1(n,j,k)|& \le C\||\nabla^3\phi| f\|^2+C\sum_{j=1}^{d}\|h^{-1}D_s^jf\|^2\\
& \le C \tau^{2}\||\nabla^3\varphi| f\|^2+C\sum_{j=1}^{d}\|h^{-1}D_s^jf\|^2.
\end{align*}

For the second four terms in \eqref{eq:comm1a}, we similarly expand as in \eqref{sinh} and
\begin{align*}
&(\cosh(\phi(n+he_j + h e_k)-\phi(n))-1)\\
&\quad = \cosh(D^j_+ \phi(n+h e_k)+D^k_+ \phi(n))-1\\
&\quad = \cosh(h \p_j \phi(n)+h \p_k \phi(n) + O(h^2|\nabla^2 \phi(y)|))\\
&\quad = 1 + \frac{1}{2} h^2 |\p_j \phi(n)+\p_k \phi(n)|^2 +
E_3(n) -1\\
&\quad = \normalcolor \frac{1}{2}h^2 |\p_j \phi(n)+\p_k \phi(n)|^2 + E_3(n),
\end{align*}
where 
\begin{align*}
|E_3(n)| 
&\leq  C(h^3|\nabla^2 \phi(y)|^{2} + |\nabla \phi(y)|^2)+Ch^4\max\{|\nabla\phi(y)|, |\nabla^2 \phi(y)|\}^4 \\
&\leq  C\left[h^{3} \tau^2 \max\{|\nabla \varphi|,|\nabla^2 \varphi|\}^2 + h^4\tau^4 \max\{|\nabla \varphi|,|\nabla^2 \varphi|\}^4\right].
\end{align*}

A similar expansion holds for the term involving $(\cosh(\phi(n-he_j - h e_k)-\phi(n))-1)$, while
$$
(\cosh(\phi(n-he_j + h e_k)-\phi(n))-1)  = \normalcolor \frac{1}{2}h^2 |\p_j \phi(n)-\p_k \phi(n)|^2 + E_3(n)
$$
with similar expansion for the term $(\cosh(\phi(n+he_j - h e_k)-\phi(n))-1)$.

Hence, expanding the contribution of the $\sinh$ just with one main term $h^2\p_{jk}\phi(n)$, i.e.
$$
\sinh(D^{j,k}_{+,-}\phi(n)) = h^2 \p_{jk}\phi(n) + O(h^6 |\nabla^2 \phi(y)|^3 + h^9 |\nabla^3 \phi(y)|^3),
$$
a multiplication of these expansions turns the second set of four terms from \eqref{eq:comm1a} into
\begin{align*}
& \frac12\p_{jk}\phi(n)\big(|\p_j \phi(n)+\p_k \phi(n)|^2(f(n+he_j)\overline{f(n+h e_k)} + f(n-h e_j) \overline{f(n-he_k)}) \\
&\qquad  - |\p_j \phi(n) - \p_k \phi(n)|^2(f(n+he_j)\overline{f(n-he_k)} + f(n-he_j)\overline{f(n+h e_k)}) \big)  \\
&\qquad + E_4(n,j,k), 
\end{align*}
where
\begin{align*}
|E_4(n,j,k)|
& \leq  C \left[h \tau^3 \max\{|\nabla \varphi| ,|\nabla^2 \varphi|,|\nabla^3\varphi| \}^3+ h^2\tau^5 \max\{|\nabla \varphi| ,|\nabla^2 \varphi|,|\nabla^3\varphi| \}^5\right.\\&\left. \quad +h^6\tau^7 \max\{|\nabla \varphi| ,|\nabla^2 \varphi|,|\nabla^3\varphi| \}^7 \right] (|f(n\pm he_j)|^2  + |f(n \pm h e_k)|^2), \\
& \leq  C \left[\delta_0 \tau^2 + \delta_0^2 \tau^3 \right] (|f(n\pm he_j)|^2  + |f(n \pm h e_k)|^2), 
\end{align*}
where we used Young's inequality $ab \le \dfrac{a^{p}}{p} + \dfrac{b^q}{q}$ for conjugate exponents $p$ and $q$. Combining the estimates for $E_1,\ E_2$ and $E_4$ we arrive at the claimed estimate for the commutator by taking into account that $h,\delta_0\in(0,1)$ and $\tau h <\delta_0$.
\end{proof}

As a next step, we freeze coefficients in the operators $\widetilde{S}_{\phi}$, $\widetilde{A}_{\phi}$ and $\mathcal{C}_{jk}^{f,f}$ when acting on functions supported in sets of the size $\epsilon_0^{-1} \tau^{-\frac{1}{2}}$ for $\epsilon_0>0$ sufficiently small and $\tau >1$ sufficiently large, both of which are to be determined below (see the proof of Theorem \ref{thm:Carl}).

\begin{lem}
\label{lem:freeze}
Let $f \in C^{\infty}_c(B_2 \setminus \overline{B_{\frac{1}{2}}})$ be such that $\supp(f)$ is of the size $\epsilon_0^{-1} \tau^{-\frac{1}{2}}$. Assume that $\phi$ is as in Theorem \ref{thm:Carl} and that $1 < \tau \leq \delta_0 h^{-1}$ for a sufficiently small constant $\delta_0>0$.
Let $\bar{n} \in \R^d$ be a point which is in the interior of $\supp(f)$ and set
\begin{align*}
 \bar{S}_{\phi}^j f(n) &:= h^{-2} \D_{d,h,j} f(n) + \frac{(\p_j \phi(\bar{n}))^2}{2}(f(n + e_j h) + f(n- e_j h)),\\
 \bar{A}_{\phi}^j f(n) &:= - h^{-1} (\p_j \phi(\bar{n}))(f(n + h e_j) - f(n - h e_j)),\\
 \bar{\mathcal{C}}_{jk}^{f,f}(n) & := (\p_{jk} \phi(\bar{n}))\big( h^{-1}(f(n + h e_j) - f(n - h e_j)) \overline{h^{-1} (f(n + h e_k) - f(n - h e_k))} \big)\\
 & \quad +\dfrac12 (\p_{jk}\phi(\bar{n}))\big( |\p_j \phi(\bar{n}) + \p_k \phi(\bar{n})|^2 (f(n + h e_j) \overline{f(n + h e_k)} + f(n - h e_j)\overline{f(n- h e_k)}) \\
& \quad - |\p_j \phi(\bar{n}) - \p_k \phi(\bar{n})|^2(f(n + h e_j)\overline{f(n - h e_k)} + f(n - h e_j)\overline{f(n + h e_k)}) \big).
\end{align*}
Then,
\begin{align*}
|\|\widetilde{S}_{\phi} f\| - \| \bar{S}_{\phi} f\| | &\leq C \tau^{\frac{3}{2}} \epsilon_0^{-1} \|f\|, \\
|\|\widetilde{A}_{\phi} f\| - \| \bar{A}_{\phi} f\| | &\leq C\tau^{\frac{1}{2}} \epsilon_0^{-1}  \sum\limits_{j=1}^{d} h^{-1}\|D_s^j f\|, \\
\sum_{n\in(h\Z)^d}\sum_{j,k}|\mathcal{C}_{jk}^{f,f} - \bar{\mathcal{C}}_{jk}^{f,f}| & \leq C \tau^{\frac{1}{2}} \epsilon_0^{-1} \sum\limits_{j=1}^d \| h^{-1}D_s^jf\|^2 + C\tau^{\frac{5}{2}} \epsilon_0^{-1} \|f\|^2.
\end{align*}
\end{lem}

\begin{proof}
Using the triangle inequality and the support condition, we estimate
\begin{align*}
|\|\widetilde{S}_{\phi} f\| - \| \bar{S}_{\phi} f\| |
&\leq C \|(\widetilde{S}_{\phi} - \bar{S}_{\phi}) f\|
\leq C ( \|((\p_j \phi(n))^2 - (\p_j \phi(\bar{n}))^2) f(n + h e_j)\| \\
& \quad + \|((\p_j \phi(n))^2 - (\p_j \phi(\bar{n}))^2) f(n - h e_j)\|)\\
& \leq C \tau^2 \sup\limits_{n \in \supp(f)}|n- \bar{n}|\|f\|
 \leq C \tau^{\frac{3}{2}} \epsilon_0^{-1} \|f\|.
\end{align*}
As the arguments for $\widetilde{A}_{\phi}$ and for $\mathcal{C}_{jk}^{f,f}$ are analogous, we do not discuss the details.
\end{proof}

Finally, as a last auxiliary step before combining all the above  ingredients into the proof of Theorem \ref{thm:Carl}, we prove a lower bound for the operators with the frozen variables.

\begin{prop}
\label{lem:lower_frozen}
Let $\bar{S}_{\phi}$, $\bar{A}_{\phi}$ and $\bar{\mathcal{C}}_{jk}^{f,f}$ be as in Lemma \ref{lem:freeze}. Then there exist $C_{\textup{low}}>0$, $c_0>0$, $h_0,\delta_0 \in (0,1)$ (small) and $ \tau_0>1$ such that for all $\tau \in ( \tau_0, \delta_0 h^{-1})$ (large), $h \in (0,h_0)$ and for all $f\in C_c^{\infty}(B_2 \setminus B_{\frac{1}{2}})$ we have
\begin{equation}
\label{eq:lower_frozen}
\|\bar{S}_{\phi} f \|^2 + \|\bar{A}_{\phi} f \|^2 + c_0 \tau \sum_{n\in(h\Z)^d}\sum\limits_{j,k=1}^{d} \bar{\mathcal{C}}_{jk}^{f,f}(n) 
 \geq C_{\textup{low}} \Big(\tau^4 \| f\|^2 + \tau^2 h^{-2} \sum\limits_{j=1}^{d}\|D^j_s f\| ^2 +  h^{-4} \sum\limits_{j=1}^{d}\|(D^j_s)^2   f\| ^2\Big).
\end{equation}
\end{prop}

\begin{proof}
Using that the operators under consideration all have constant coefficients, we may perform a Fourier transform and infer that
\begin{align}
\label{eq:lower_1}
\begin{split}
& \|\bar{S}_{\phi} f \|^2 + \|\bar{A}_{\phi} f \|^2 + c_0\tau\sum_{n\in(h\Z)^d}\sum_{j,k=1}^d\bar{\mathcal{C}}_{jk}^{f,f}\\
& =  \sum\limits_{j = 1}^d \| \big[ -4 h^{-2} \sin^2(h \xi_j/2 ) + (\p_j \phi(\bar{n}))^2 \cos(\xi_j h)  \big] \widehat{f} \|^2
+  \sum\limits_{j=1}^d  \|2h^{-1}(\p_j \phi(\bar{n}))\sin(\xi_j h) \widehat{f}\|^2 \\
& \quad + c_0\tau \sum\limits_{j,k=1}^d  \left[4  ( h^{-1}\sin(\xi_j h) (\p_{jk}\phi(\bar{n})) h^{-1}\sin(\xi_k h)\widehat{f}, \widehat{f}) \right.\\
& \quad \left.
+ \left(\left( \p_{jk}\phi(\bar{n})\right)\left(|\p_j \phi(\bar{n}) + \p_k \phi(\bar{n})|^2 \cos(h \xi_j - h \xi_k) - |\p_j \phi(\bar{n}) - \p_k \phi(\bar{n})|^2 \cos(h \xi_j + h \xi_k) \right) \widehat{f}, \widehat{f}\right)
\right].
\end{split}
\end{align}In order to prove the positivity of this expression, we will choose $c_0>0$ so small, that outside of a sufficiently small neighbourhood of the union of the (joint) characteristic sets of the Fourier symbols
\begin{align*}
p_{r,j}(\xi) &:= - 4 h^{-2} \sin^2(h \xi_j/2 ) + (\p_j \phi(\bar{n}))^2 \cos(\xi_j h) \\
& =  2 h^{-2}(\cos(h \xi_j)-1) + (\p_j \phi(\bar{n}))^2 \cos(\xi_j h),\\
p_{i,j}(\xi) &:= 2 \p_j \phi(\bar{n}) h^{-1} \sin(\xi_j h),
\end{align*}
the third term in \eqref{eq:lower_1} is controlled by these. In order to observe that this is possible, we first study the contributions $p_{r,j}$ and $p_{i,j}$ separately. We first consider the terms $p_{r,j}$ and $p_r(\xi):= \sum\limits_{j=1}^{d} p_{r,j}(\xi) $ associated with the symmetric operator. We begin by observing that the first summand in
\begin{equation}
\label{ps}
p_r(\xi)=\sum\limits_{j=1}^{d}|\p_j \phi(\bar{n})|^2 \cos(\xi_j h) +2\sum\limits_{j=1}^{d} \frac{\cos(\xi_j h)-1}{h^2}
\end{equation}
is bounded from above by $C \tau^2$. For the second summand, we deduce that, since $|\cos(x)|\in (0,1)$ and for $\xi_j\in h^{-1}(-\pi, \pi)$, we have
$$
|\cos(\xi_jh)-1|=\Big|\frac{(\xi_jh)^2}{2}+R(\xi_jh)\Big|\ge (\xi_jh)^2\Big(\frac12-\frac{\pi^2}{24}\Big)\ge \frac{1}{16}(\xi_jh)^2,
$$
where $R(\xi_jh)$ is the remainder term in the Taylor approximation. Hence,
\begin{equation}
\label{eq:elliptic_high_frequ}
\Big| \sum\limits_{j=1}^d \frac{\cos(\xi_j h) - 1}{h^2} \Big|
 = h^{-2} \sum\limits_{j=1}^d |\cos(\xi_j h) - 1|
 \geq \frac{1}{16} h^{-2} \sum\limits_{j=1}^d |\xi_j h|^2 \geq  \frac{1}{16}|\xi|^2.
\end{equation}
Combining these two observations, we note that there exists a constant $C_1>0$ such that if $|\xi|\geq C_1 \tau$, the expression in \eqref{ps} can be estimated from below by
\begin{align}
\label{pr}
\notag p_r(\xi)^2\ge \Big|\sum\limits_{j=1}^{d} \frac{\cos(\xi_j h)-1}{h^2}\Big|^2\ge3 c_{hf}|\xi|^4&\ge c_{hf}(|\xi|^4+\tau^2|\xi|^2 + \tau^4) \\
&\geq c_{hf}\Big(\sum\limits_{j=1}^{d} h^{-4}\sin^4(h\xi_j)+\tau^2\sum\limits_{j=1}^{d} h^{-2}\sin^2(h\xi_j) + \tau^4\Big).
\end{align}
Here the constant $c_{hf}>0$ is independent of $\tau$ and $\xi$.
In the sequel, this will motivate a distinction between the two regimes $|\xi|\geq C_1 \tau$ and $|\xi|\leq C_1 \tau$.
We further note that if the constant $c_0>0$ in \eqref{eq:lower_1} is sufficiently small, then the a priori not necessarily signed Fourier multipliers associated with contributions in the third and fourth line in \eqref{eq:lower_1} may be absorbed into the lower bound in \eqref{eq:elliptic_high_frequ}.
Motivated by the estimate \eqref{eq:elliptic_high_frequ}, we call the region $\{|\xi|\geq C_1 \tau\}$ the \emph{high frequency elliptic region}. By the above considerations the claimed lower bound \eqref{eq:lower_frozen} always holds in this region.

It thus remains to study the region complementary to this, i.e. the region in which $|\xi|\leq C_1 \tau$. In this region, we expand the symbols in $h\xi_j$ (noting that $h|\xi| \leq C_1 \tau \delta_0 \tau^{-1} = C_1 \delta_0$ which is small for $\delta_0>0$ small). For the symmetric part we obtain for some constant $C>0$ which depends on $C_1>0$
\begin{align}
\label{eq:approx_sym}
\begin{split}
\big| p_r(\xi)- \sum\limits_{j=1}^{d} \big( |\p_j \phi(\bar{n})|^2 - |\xi_j|^2  \big) \big|
&\leq C \sum\limits_{j=1}^{d}\big( |\p_j \phi(\bar{n})|^2|\xi_j h|^2 + h^{-2}|\xi_j h|^4 \big)\\
& \leq C \tau^2 h^2 |\xi|^2 |\nabla \varphi(\bar{n})|^2 + h^2 |\xi|^4 \leq C (\tau^4 h^2 |\nabla \varphi(\bar{n})|^2 + h^2 \tau^4).
\end{split}
\end{align}
For the antisymmetric part in turn we infer for $p_i:=\sum_{j=1}^dp_{i,j}$,
\begin{align}
\label{eq:approx_anti}
\begin{split}
\big|p_{i}(\xi)- 2\sum_{j=1}^d \p_j \phi(\bar{n})\xi_j\big| \leq C \tau |\nabla \varphi(\bar{n})|h^{-1}\sum_{j=1}^{d} |h\xi_j|^3 \leq C h^2 \tau^4 |\nabla\varphi(\bar{n})|.
\end{split}
\end{align}
Let now
\begin{equation}
\label{eq:char_set}
\mathcal{C}_{\tau}:=\{\tau^2 |\nabla \varphi(\bar{n})|^2 = |\xi|^2\}\cap \{\tau \nabla \varphi(\bar{n})\cdot \xi =0\},
\end{equation}
denote the joint characteristic sets of the symmetric and antisymmetric parts of the operator. Further define
\begin{equation*}
\mathcal{N}_{\tau, \mathcal{C}}:= \{ \xi \in (h^{-1} \mathbb{T})^d: \ \dist(\xi, \mathcal{C}_{\tau}) \leq \gamma_0 \tau \}
\end{equation*}
to be a $\gamma_0 \tau$ neighbourhood of the joint characteristic set $\mathcal{C}_{\tau}$ with $\gamma_0>0$ small (to be determined below).
With this notation fixed, we prove that for $|\xi|\leq C_0 \tau$ outside of $\mathcal{N}_{\tau, \mathcal{C}}$
there exists some constant $c_{lf,1}>0$ (depending on $\gamma_0$) independent of $\tau>0$ such that
\begin{equation}
\label{eq:charactistic_1}
p_r^2(\xi) + p_i^2(\xi) \geq c_{lf,1}(\tau^4 + |\xi|^4).
\end{equation}
Indeed, this is true for the leading order approximations
\begin{equation*}
(|\nabla \phi(\bar{n})|^2 - |\xi|^2)^2 + 4(\nabla \phi(\bar{n})\cdot \xi)^2,
\end{equation*}
and transfers to the full symbols $p_r^2(\xi) + p_i^2(\xi)$ since the error estimates in \eqref{eq:approx_sym}, \eqref{eq:approx_anti} are of order $C h^2 \tau^4 \leq C \delta_0^2 \tau^2$ if $\tau \in (1,\delta_0 h^{-1})$. Thus, if $\delta_0$ is sufficiently small (depending on $\gamma_0$), these error contributions can be absorbed into the right hand side of \eqref{eq:charactistic_1}. Again, if the constant $c_0>0$ is sufficiently small, we may absorb the contributions originating from the not necessarily signed Fourier symbols of the operators in the third and fourth line in \eqref{eq:lower_1} into the lower bound \eqref{eq:charactistic_1}.

It remains to study the behaviour of the Fourier symbols associated to the operators from \eqref{eq:lower_frozen} in the neighbourhood $\mathcal{N}_{\tau, \mathcal{C}}$ of the joint characteristic set \eqref{eq:char_set}. To this end, we also carry out an expansion of the symbol associated with the operators in the third and fourth line of \eqref{eq:lower_1} (which originates from the commutator) and obtain the symbol $q(\xi) = q_1(\xi) + q_2(\xi)$ with
\[
q_1(\xi) =\sum_{j,k=1}^d 4\tau \p_{jk}\phi(\bar{n}) \xi_j \xi_k +\tau h^{2}O(|\nabla^2\phi(\bar{n})| |\xi|^4) = \sum_{j,k=1}^d4\tau^2 \p_{jk}\varphi(\bar{n}) \xi_j \xi_k +\tau^6 h^2O(|\nabla^2\varphi(\bar{n})|),
\]
and
\begin{align*}
q_2(\xi) &=
\sum_{j,k=1}^d 4 \tau^4 \p_{jk}\varphi(\bar{n})\p_j \varphi(\bar{n})\p_k \varphi(\bar{n}) + \tau^4 O(|\nabla \varphi(\bar{n})|^2 |\nabla^2 \varphi(\bar{n})| |h\xi|^2) \\
 & = \sum_{j,k=1}^d4 \tau^4 \p_{jk}\varphi(\bar{n})(\p_j \varphi(\bar{n}))(\p_k \varphi(\bar{n}))+  \tau^6 h^{2} O(|\nabla \varphi(\bar{n})|^2 |\nabla^2 \varphi(\bar{n})|).
\end{align*}
Using that $\tau \in (1,\delta_0 h^{-1})$, we thus obtain that
\begin{equation*}
q(\xi) = 4 (\tau^2 \xi \cdot \nabla^2 \varphi(\bar{n}) \xi + \tau^4 \nabla \varphi(\bar{n})\cdot \nabla^2 \varphi(\bar{n}) \nabla \varphi(\bar{n}) ) + O(C \delta_0^2 \tau^4  ).
\end{equation*}
Now by the pseudoconvexity conditions on $\phi$ for $\bar{n}\in B_2 \setminus B_{\frac{1}{2}}$ (see Lemma \ref{lem:weight}), we infer that for $\xi $ in the characteristic set \eqref{eq:char_set} there exist constants $c_{cf,1}, c_{cf}>0$ which are independent of $\tau$ and $h$ such that
\begin{equation*}
q(\xi) \geq c_{cf,1}(\tau^4 + |\xi|^2 \tau^2+|\xi|^4) - C \delta_0^2 \tau^4
\geq c_{cf}(\tau^4 + |\xi|^2 \tau^2+|\xi|^4).
\end{equation*}
We next seek to argue that by continuity a similar lower bound also holds on $\mathcal{N}_{\tau, \mathcal{C}}$. To this end, note that for $\xi \in \mathcal{N}_{\tau, \mathcal{C}}$ we have $\xi = \tau \xi_0$ for some $\xi_0 \in (h^{-1} \mathbb{T})^d$ with $|\xi_0| \in (C_{0,1}, C_{0,2})$, where the constants $C_{0,1}, C_{0,2}>0$ only depend on $\gamma_0$ and the dimension $d$ and, in particular, are independent of $\tau>1$ and $h>0$. Thus, for $\xi \in \mathcal{N}_{\tau, \mathcal{C}}$ and $\xi_0 = \tau^{-1} \xi$ we have that by homogeneity
\begin{equation*}
\widetilde{q}(\xi):= \tau^{-4}q(\xi)=\ \xi_0 \cdot \nabla^2 \varphi(\bar{n}) \xi_0 +  \nabla \varphi(\bar{n})\cdot \nabla^2 \varphi(\bar{n}) \nabla \varphi(\bar{n})
\end{equation*}
is independent of $\tau$. Since for $\xi \in \mathcal{C}_{\tau}$ the pseudoconvexity condition for $\phi$ implies that $\widetilde{q}(\xi) \geq c_{cf,1}>0$, by continuity, it remains true that $\widetilde{q}(\xi)\geq c_{cf,1}/2$ in the neighbourhood $\mathcal{N}_{\tau, \mathcal{C}}$ if $\gamma_0>0$ is sufficiently small (but independent of $\tau>1$). By the scaling of $q(\xi)$ we thus infer that for $\xi \in \mathcal{N}_{\tau, \mathcal{C}}$ and $\delta_0>0$ sufficiently small we have
\begin{equation}
\label{q}
q(\xi) \geq \frac{c_{cf,1}}{2}(\tau^4 + |\xi^2| \tau^2) - C \delta_0^2 \tau^4
\geq \frac{c_{cf}}{4}(\tau^4 + |\xi|^2 \tau^2+|\xi|^4).
\end{equation}
Thus, in total, by \eqref{pr}, \eqref{eq:charactistic_1} and \eqref{q}, we have obtained that for all $\xi \in (h^{-1} \mathbb{T})^d$
\begin{equation*}
p_r^2(\xi) + p_i^2(\xi) + q(\xi) \geq \min\{c_{cf}/4, c_{lf,1}, c_{hf}\}\Big(\tau^4 + \tau^2 h^{-2} \sum\limits_{j=1}^d \sin^2(h \xi_j) + h^{-4}\sum\limits_{j=1}^d \sin^4(h \xi_j) \Big).
\end{equation*}
By the Parseval identity, this implies that
\begin{align*}
\begin{split}
& \|\bar{S}_{\phi} f \|^2 + \|\bar{A}_{\phi} f \|^2 + c_0\tau\sum_{n\in(h\Z)^d}\sum_{j,k=1}^d\bar{\mathcal{C}}_{jk}^{f,f} \\
&\quad\geq C_{\text{low}} (\tau^4 \| f\|^2 +  h^{-4}\sum\limits_{j=1}^{d}\| (D^j_s)^2  f\| ^2 + \tau^2 h^{-2}\sum\limits_{j=1}^{d}\|D^j_s   f\| ^2) ,
\end{split}
\end{align*}
which yields the claim of the Proposition. 
\end{proof}

With all of these auxiliary results in hand, we now address the proof of Theorem \ref{thm:Carl}.

\begin{proof}[Proof of Theorem \ref{thm:Carl}]
The proof of Theorem \ref{thm:Carl} follows by combining all the previous estimates. We first rewrite the desired estimate in terms of the functions $f := e^{\phi}u$ for which we seek to prove
\begin{equation*}
\tau^{\frac{3}{2}}\|f\| + \tau^{\frac{1}{2}} \| h^{-1} D_s f\| + \tau^{-\frac{1}{2}} \| h^{-2} D^2_s f\| \leq C \|L_{\phi} f\|
\end{equation*}
(and for which we note that the action of $D_s$ on $e^{\phi}u$ yields terms $D_s e^{\phi}$ that can be absorbed in the first term with $\|e^{\phi}u\|$).
We now argue in two steps, first reducing the estimate to a bound for the localized functions and then proving the estimate for these.

\emph{Step 1: Localization.}
As a first step, we note that it suffices to prove the estimate
\begin{equation}
\label{eq:Carl_aux}
\tau^{\frac{3}{2}}\|f\| + \tau^{\frac{1}{2}} \| h^{-1} D_s f\| + \tau^{-\frac{1}{2}} \| h^{-2} D^2_s f\| \leq C \|L_{\phi} f\|
\end{equation}
for the localized functions $f_k$ from Lemma \ref{lem:localize}. Indeed, assuming that the estimate \eqref{eq:Carl_aux} is proven for $f_k$, an application of Minkowski's inequality and the error estimates from Lemma \ref{lem:localize} yield
\begin{align}
\label{eq:Carl_aux1}
\begin{split}
&\tau^{\frac{3}{2}}\|f\| + \tau^{\frac{1}{2}} \| h^{-1} D_s f\| + \tau^{-\frac{1}{2}} \| h^{-2} D^2_s f\|\\
&\quad \leq
\tau^{\frac{3}{2}}\sum_{k} \| f_k \| + \tau^{\frac{1}{2}} \sum_{k}\| h^{-1} D_s f_k\| + \tau^{-\frac{1}{2}} \sum_{k} \| h^{-2} D^2_s f_k\|\\
&\quad  \leq C \sum\limits_k \|L_{\phi} f_k\|
\leq C \|L_{\phi} f\| +  C_{\text{loc}} \tau^{\frac{1}{2}} \epsilon_0 \sum\limits_{j=1}^{d}\|h^{-1}D^j_s f\| + C_{\text{loc}} (\tau \epsilon_0 +\tau^{\frac{3}{2}} \epsilon_0+ \tau^2 \tau^{\frac{1}{2}} h \epsilon_0)  \|f\|.
\end{split}
\end{align}
Now choosing 
\begin{align}
\label{eq:epsilon0}
\epsilon_0 \leq \frac{1}{10C_{\text{loc}}},
\end{align}
and recalling that $\tau h \leq \delta_0 $ for some $\delta_0 \in (0,1)$, we may absorb the contribution on the right hand side of \eqref{eq:Carl_aux1} into its left hand side (in particular we note that $\tau^2 \tau^{\frac{1}{2}} h  \leq \delta_0 \tau^{\frac{3}{2}}$ by our assumptions on the relation between $\tau$ and $h$). This then yields the estimate \eqref{eq:Carl_aux}. The estimate \eqref{eq:Carl_main} follows from this by possibly choosing the constants in the terms which involve derivatives on the left hand side of \eqref{eq:Carl_aux} smaller, carrying out the product rule and absorbing the $L^2$ errors into the $L^2$ contribution on the left hand side of \eqref{eq:Carl_aux}.\\

\emph{Step 2. Proof of \eqref{eq:Carl_aux} for the localized functions.} It thus suffices to prove \eqref{eq:Carl_aux} for $f= f_k$.
To this end, we observe that for $f_k = f \psi_k$ with $\supp(f) \subset B_2 \setminus B_{1/2}$, $\psi_k$ as in Lemma \ref{lem:localize}, $c_0 \in (0,1)$ as in Proposition \ref{lem:lower_frozen} and $\tau_0>1$ such that $\tau_0 c_0 \geq 1$,
\begin{align*}
\begin{split}
\tau \|L_{\phi} f_k \|^2
& = \tau \|S_{\phi} f_k \|^2 + \tau \|A_{\phi} f_k \|^2 +  \tau (f_k ,[S_{\phi}, A_{\phi} ] f_k )\\
& \geq \|S_{\phi} f_k \|^2 + \|A_{\phi} f_k \|^2 + \tau c_0 (f_k ,[S_{\phi}, A_{\phi} ] f_k )\\
& = \|\widetilde{S}_{\phi} f_k \|^2 + \|\widetilde{A}_{\phi} f_k \|^2 + \tau c_0 \sum_{n\in(h\Z)d}\sum\limits_{j_1, j_2=1}^{d}\mathcal{C}_{j_1j_2}^{f_k,f_k}- E_1,
\end{split}
\end{align*}
where by Lemma \ref{lem:discrete_approx}, taking into account that $0<\delta_0<1<\tau$, we get
\begin{align}
\label{eq:E1}
|E_1| \leq C (\delta_0^2\tau^4+\tau^{3}) \|f_k\|^2 + C \tau  \sum\limits_{j=1}^d \|h^{-1}D^j_s f_k \|^2.
\end{align} 
Choosing $\delta_0>0$ such that $C \delta_0 \leq \frac{C_{\text{low}}}{10}$, where $C_{\text{low}}$ is the constant from Proposition \ref{lem:lower_frozen}, we will be able to treat the contributions in \eqref{eq:E1} as error contributions in the following arguments.

Exploiting the bounds from Lemma \ref{lem:freeze}, we may further estimate
\begin{multline*}
\|\widetilde{S}_{\phi} f_k \|^2 + \|\widetilde{A}_{\phi} f_k \|^2  +\tau c_0 \sum_{n\in(h\Z)d}\sum\limits_{j_1, j_2=1}^{d}\mathcal{C}_{j_1j_2}^{f_k,f_k} - E_1 \\
 \geq
 \|\bar{S}_{\phi} f_k \|^2 +\|\bar{A}_{\phi} f_k \|^2 +  \tau c_0 \sum_{n\in(h\Z)d}\sum\limits_{j_1, j_2=1}^{d}\bar{\mathcal{C}}_{j_1j_2}^{f_k,f_k}- E_1 - E_2 ,
\end{multline*}
where by the estimates from Lemma \ref{lem:freeze}
\begin{equation}
\label{eq:E3}
|E_2| \leq C(\tau^{3}\epsilon_0^{-2} + \tau^{\frac{7}{2}}\epsilon_0^{-1}) \|f_k \|^2 + C(\tau^{\frac{3}{2}} \epsilon_0^{-1} + \tau \epsilon_0^{-2})\sum\limits_{j=1}^d \|h^{-1}D^j_sf_k\|^2.
\end{equation}

Finally, invoking Proposition \ref{lem:lower_frozen}, we infer that
\begin{align}
\label{eq:prop_frozen}
\begin{split}
&
 \|\bar{S}_{\phi} f_k \|^2 +  \|\bar{A}_{\phi} f_k \|^2 + \tau c_0 \sum_{n\in(h\Z)d}\sum\limits_{j_1, j_2=1}^{d}\bar{\mathcal{C}}_{j_1j_2}^{f_k,f_k} - E_1 - E_2\\
& \quad \geq C_{\text{low}}\Big(\tau^4 \| f_k \|^2 +  h^{-4}\sum_{j=1}^d\|(D^j_s)^2 f_k \| ^2 + \tau^2 h^{-2}\sum_{j=1}^d\|D^j_s f_k\| ^2\Big) - E_1 - E_2.
\end{split}
\end{align}
Recalling the condition for $\epsilon_0>0$ from \eqref{eq:epsilon0}, we now choose $\epsilon_0 = \frac{1}{20 C_{\text{loc}}}$ and fix $\tau_0>1$ so large and $\delta_0>0$ so small that
\begin{equation*}
C\max\{\tau_0^3,\delta_0^2 \tau_0^4, \tau_0^{\frac{7}{2}}\epsilon_0^{-1}, \tau_0^{3} \epsilon_0^{-2}\} \leq C_{\text{low}}\frac{\tau_0^4}{10}\quad  \mbox{ and }\quad C\max\{\tau_0^{\frac{3}{2}} \epsilon_0^{-1}, \tau_0 \epsilon_0^{-2} \}\leq C_{\text{low}}\frac{\tau_0^2}{10}.
\end{equation*}
Further, we choose the value of $h_0> 0$ so small that $\delta_0 h_0^{-1} \geq 100 \tau_0> 100$, which in particular implies that for all $h\in (0,h_0)$ the interval $ (\tau_0,\delta_0 h^{-1})$ is non-empty. With these choices, it follows that for $\tau \in (\tau_0,\delta_0 h^{-1})$, we may absorb the error contributions $E_1$ and $E_2$ from \eqref{eq:E1} and \eqref{eq:E3} into the positive right hand side contributions in \eqref{eq:prop_frozen}.
Therefore, we obtain that
\begin{equation*}
\tau \|L_{\phi} f_k\|^2 \geq \frac{C_{\text{low}}}{2} \Big( \tau^4 \| f_k \|^2 +  h^{-4}\sum_{j=1}^d\|(D^j_s)^2 f_k \| ^2 + \tau^2 h^{-2}\sum_{j=1}^d\|D^j_s f_k \| ^2 \Big).
\end{equation*}
Dividing by $\tau>\tau_0$ implies the desired result.
\end{proof}

\section{Proofs of Theorems \ref{thm:log_conv} and \ref{thm:3spheres} }
\label{sec:proofs_main}

In this section we provide the proofs of the results of Theorems \ref{thm:log_conv} and \ref{thm:3spheres}.

\subsection{Derivation of Theorem \ref{thm:3spheres} from Theorem \ref{thm:log_conv}}

We first show how Theorem \ref{thm:log_conv} implies Theorem \ref{thm:3spheres}.

\begin{proof}[Proof of Theorem \ref{thm:3spheres}]
 Let us assume that Theorem  \ref{thm:log_conv} holds.
First, let us take the value $\tau^*$ such that $(c_1 + c_2) \tau^* =\log \frac{\|u\|_{L^2(B_2)}}{\|u\|_{L^2(B_{1/2})}} $. It is easy to check that with this value of $\tau^*$ it holds
$$
e^{c_1\tau^*} \|u\|_{L^2(B_{1/2})} = e^{-c_2 \tau^*}\|u\|_{L^2(B_2)}.
$$
Given $u$ satisfying \eqref{eq:log_conv}, we can assume that $\tau_0<\tau^*$, and we are in one of the following two cases:
\begin{itemize}
\item If $\tau^* \in (\tau_0, \delta_0 h^{-1})$, then plugging this into the right hand side of \eqref{eq:log_conv} yields, for $\tau=\tau^*$, that
\begin{equation*}
C(e^{c_1\tau} \|u\|_{L^2(B_{1/2})} + e^{-c_2\tau}\|u\|_{L^2(B_2)})
= 2 C\|u\|^{\frac{c_2}{c_1 + c_2}}_{L^2(B_{1/2})}\|u\|_{L^2(B_2)}^{\frac{c_1}{c_1 + c_2}}.
\end{equation*}
\item If $\tau^* \notin (\tau_0, \delta_0 h^{-1})$, we observe that $\tau< \delta_0 h^{-1}<\tau^*$.
 We hence obtain that
 \begin{equation}
 \label{estim}
  e^{c_1 \frac{\delta_0}{2} h^{-1}}\|u\|_{L^2(B_{1/2})}\le  e^{c_1 \tau^*}\|u\|_{L^2(B_{1/2})}= e^{-c_2 \tau^*}\|u\|_{L^2(B_{2})}  \leq e^{- c_2 \frac{\delta_0}{2}  h^{-1}}\|u\|_{L^2(B_2)}.
 \end{equation}
Thus, since \eqref{eq:log_conv} holds for all $\tau \in (\tau_0, \delta_0 h^{-1})$ and by using \eqref{estim}, we have
 \begin{equation*}
 \|u\|_{L^2(B_1)} \le C(e^{c_1 \frac{\delta_0}{2} h^{-1}} \|u\|_{L^2(B_{1/2})} + e^{-c_2 \frac{\delta_0}{2} h^{-1} } \|u\|_{L^2(B_2)})\le 2\cdot e^{-c_2 \frac{\delta_0}{2} h^{-1} }\|u\|_{L^2(B_2)}.
\end{equation*}

  \end{itemize}
Combining both cases implies \eqref{thm:3spheres} with $\alpha = \frac{c_2}{ c_1 + c_2}$ and $c_0 = \frac{\delta_0}{2}c_2$.
\end{proof}

\subsection{Derivation of Theorem \ref{thm:log_conv} from the Carleman estimate of Theorem \ref{thm:Carl}}

In this section, we deduce Theorem \ref{eq:log_conv} from Theorem \ref{thm:Carl}. As an auxiliary result we deduce a Caccioppoli inequality for more general second order difference equations. In particular this applies to the difference Schrödinger equation \eqref{eq:Schroedinger}. 

\begin{lem}[Caccioppoli]
\label{lem:Caccioppoli}
Let $a_{jk}: (h\Z)^d \rightarrow \R^{d \times d}$ be symmetric, bounded and uniformly elliptic with ellipticity constant $\lambda \in (0,1)$, i.e. assume that for all $\xi \in \R^{d}\setminus \{0\}$ we have
\begin{equation*}
\lambda |\xi|^2 \leq \sum_{i,j=1}^d \xi_i a_{ij} \xi_j \leq \lambda^{-1} |\xi|^2.
\end{equation*}
 Let $V: (h\Z)^d \rightarrow \R$ be uniformly bounded in $h$ and $B:(h\Z)^d \rightarrow \R^d$ be a uniformly bounded tensor field. Denote $B:=(B_j)_{j=1}^d$. Let $u: (h \Z)^d \rightarrow \R$ be a weak solution of 
$$
 \Big( h^{-2}\sum\limits_{j,k=1}^{d} a_{jk}(n)D_{+,j}^
h D_{-,k}^h  + h^{-1}\sum\limits_{j=1}^d B_j(n)D_{+,j}^h  + V(n)\Big) u(n)  =0,
$$
 in the sense that $u \in H^1_{\operatorname{loc},h}((h\Z)^d)$ and for all $v\in H^1((h\Z)^d)$ with $\supp(v)$ bounded, we have
\begin{multline*}
 \sum\limits_{n \in (h \Z)^d}\Big[ h^{-2}\sum\limits_{j,k=1}^{d} a_{jk}(n)(u(n + h e_j)- u(n))(v(n + h e_k) - v(n)) \\
 - h^{-1}\sum\limits_{j=1}^d B_j(n)(u(n+h e_j)- u(n)) v(n)  - V(n) u(n) v(n) \Big] =0.
\end{multline*}
Let $0<10h<r_1< r_1 + 100 h < r_2$.
Then there exists a constant $C>1$ depending on $r_1, r_2, \|V\|_{L^{\infty}}, \|B\|_{L^{\infty}}$ such that
\begin{equation*}
\sum\limits_{j=1}^d\| h^{-1}( u(\cdot +h e_j)- u(\cdot))\|_{L^2(B_{r_1})}^2 \leq C \| u\|_{L^2(B_{r_2})}^2 .
\end{equation*}
\end{lem}

Here $H^{1}_{\operatorname{loc},h}((h \Z)^d)$ and $H^{1}((h \Z)^d)$ denote the local and global $H^1$ spaces on the lattice.

\begin{proof}[Proof of Lemma \ref{lem:Caccioppoli}]
The result follows along the same lines as the continuous Caccioppoli inequality; we only present the proof for completeness. As for general $r_1, r_2$ the proof is analogous, we only discuss the details in the case $r_1 = 1$, $r_2 =2$ and $0 < h \leq h_0$ for $h_0 \ll 1$ sufficiently small.

Let $\eta: (h \Z)^d \rightarrow \R$ be a cut-off function which is equal to one on $B_1$ and vanishes outside of $B_2$.
The function $(u \eta^2)(n)$ is then an admissible test function in the Schrödinger equation for $u$. Inserting this, we obtain
\begin{multline*}
0 = \sum\limits_{n \in (h \Z)^d}\Big[ h^{-2}\sum\limits_{j,k=1}^{d} a_{jk}(n)(u(n + h e_j)- u(n))((u\eta^2)(n + h e_k) - (u \eta^2)(n))\\
 - h^{-1}\sum\limits_{j=1}^d B_j(n)\big(u(n+h e_j)- u(n)\big) (u \eta^2)(n)  - V(n) u(n) (u \eta^2)(n) \Big].
\end{multline*}
We first deal with the leading, second order contribution. Noting that
\begin{equation*}
(u\eta^2)(n + h e_k) - (u \eta^2)(n)
= \big(u(n + h e_k)- u(n)\big) \eta^2(n) + u(n + h e_k )\big(\eta^2 (n + h e_k)-\eta^2(n)\big),
\end{equation*}
we obtain that
\begin{align}
\label{eq:product_rule}
\begin{split}
&\sum\limits_{n \in (h \Z)^d} h^{-2}\sum\limits_{j,k=1}^{d} a_{jk}(n)(u(n + h e_j)- u(n))((u\eta^2)(n + h e_k) - (u \eta^2)(n))\\
&\quad= \sum\limits_{n \in (h \Z)^d} h^{-2}\sum\limits_{j,k=1}^{d} a_{jk}(n)(u(n + h e_j)- u(n))(u(n + h e_k) - u(n))\eta^2(n)\\
& \qquad + \sum\limits_{n \in (h \Z)^d} h^{-2}\sum\limits_{j,k=1}^{d} a_{jk}(n)(u(n + h e_j)- u(n))u(n + h e_k )(\eta^2 (n + h e_k)-\eta^2(n)).
\end{split}
\end{align}
By virtue of the ellipticity of $a_{jk}$ we further infer that
\begin{align*}
\begin{split}
&\sum\limits_{n \in (h \Z)^d}h^{-2}\sum\limits_{j,k=1}^{d} a_{jk}(n)(u(n + h e_j)- u(n))(u(n + h e_k) - u(n))\eta^2(n)\\
&\quad \geq \lambda \sum\limits_{n \in (h \Z)^d}h^{-2}\sum\limits_{j=1}^{d} (u(n + h e_j)- u(n))^2\eta^2(n) = \lambda \sum\limits_{j=1}^{d} \|h^{-1}(u(\cdot + h e_j)- u(\cdot)) \eta \|^2_{L^2((h \Z)^d)}.
\end{split}
\end{align*}
For the second contribution on the right hand side of \eqref{eq:product_rule}, we rewrite $\eta^2(n+he_k)-\eta^2(n) = (\eta(n + he_k)-\eta(n))(\eta(n + he_k)+ \eta(n))$ and estimate from above:
\begin{align}
\label{eq:error_term_Cac}
\begin{split}
&h^{-2} a_{jk}(n)(u(n + h e_j)- u(n))u(n + h e_k )(\eta^2 (n + h e_k)-\eta^2(n))\\
&= h^{-2} a_{jk}(n)(u(n + h e_j)- u(n))u(n + h e_k )(\eta(n + he_k)-\eta(n))(\eta(n + he_k)+ \eta(n))\\
& = h^{-2} a_{jk}(n)(u(n + h e_j)- u(n))\eta(n) u(n + h e_k )(\eta(n + he_k)-\eta(n))\\
& \quad + h^{-2} a_{jk}(n)(u(n + h e_j)- u(n))\eta(n + he_k) u(n + h e_k )(\eta(n + he_k)-\eta(n))\\
& = 2h^{-2} a_{jk}(n)(u(n + h e_j)- u(n))\eta(n) u(n + h e_k )(\eta(n + he_k)-\eta(n))\\
& \quad + h^{-2} a_{jk}(n)(u(n + h e_j)- u(n))(\eta(n+ h e_k)- \eta(n )) u(n + h e_k )(\eta(n + he_k)-\eta(n))\\
&\le h^{-2} a_{jk}(n)(u(n + h e_j)- u(n))^2 \eta^2(n) + h^{-2} a_{jk}(n) u^2(n + h e_k )(\eta(n + he_k)-\eta(n))^2\\
&\quad + C h^{-2} a_{jk}(n)(u^2(n + h e_j)  +u^2(n)+ u^2(n + h e_k ))  (\eta(n + he_k)-\eta(n))^2 .
\end{split}
\end{align}
Noting that  $a_{ij} \le \frac{1}{2}\lambda^{-1}$ (this follows from the ellipticity condition when choosing appropriate $\xi$) we obtain that
\begin{align*}
\begin{split}
&\sum\limits_{n \in (h \Z)^d} h^{-2}  \sum_{j,k=1}^n a_{jk}(n)(u(n + h e_j)- u(n))u(n + h e_k )(\eta^2 (n + h e_k)-\eta^2(n))\\
&  \leq\sum\limits_{n \in (h \Z)^d} \dfrac{1}{2}\lambda^{-1} \sum_{j=1}^d ( \eta^{2}(n)h^{-2}(u(n + h e_j)-u(n))^2) + C_{\lambda} \|u\|_{L^2(B_2)}^2 \sup\limits_{k}|h^{-1}(\eta(n+ h e_k)- \eta(n))|^2.
\end{split}
\end{align*}

Combining this with the bounds for $B_j$ and $V$, we obtain
\begin{align}
\label{eq:Cacc_aux_comb}
\begin{split}
&\lambda \sum\limits_{j=1}^{d} \|h^{-1}(u(\cdot + h e_j)- u(\cdot)) \eta \|^2_{L^2((h \Z)^d)}\leq \frac{\lambda^{-1}}{2}\sum\limits_{j=1}^{d} \|h^{-1}(u(\cdot + h e_j)- u(\cdot)) \eta \|^2_{L^2((h \Z)^d)} \\
& \quad +  C_{\lambda}\sup\limits_{k}\|h^{-1}(\eta(\cdot+ h e_k)- \eta(\cdot))\|_{L^{\infty}((h\Z)^d)}^2 \|u\|_{L^2(B_2)}^2 \\
& \quad + \|V\|_{L^{\infty}(B_2)} \|u\|_{L^2(B_2)}^2
 + \|B\|_{L^{\infty}(B_2)} \|h^{-1}(u(\cdot + h e_j) - u(\cdot )) \eta\|_{L^2((h \Z)^d)}\|u\|_{L^2(B_2)}.
 \end{split}
\end{align}
Here the first contribution in \eqref{eq:Cacc_aux_comb} originates from the first right hand side contribution in \eqref{eq:error_term_Cac}. We may absorb it from the right hand side of \eqref{eq:Cacc_aux_comb} into the left hand side of \eqref{eq:Cacc_aux_comb}. Using Young's inequality for the contribution
\begin{align*}
\begin{split}
& \|B\|_{L^{\infty}(B_2)} \|h^{-1}(u(\cdot + h e_j) - u(\cdot )) \eta\|_{L^2((h \Z)^d)} \|u\|_{L^2(B_2)}
\\
&\leq \frac{\lambda}{4}\|h^{-1}(u(\cdot + h e_j) - u(\cdot )) \eta\|^2_{L^2((h \Z)^d)}+ C_{\lambda}\|B\|^2_{L^{\infty}(B_2)} \|u\|_{L^2(B_2)}^2,
\end{split}
\end{align*}
allows us to also absorb the gradient term in this contribution into the left hand side of \eqref{eq:Cacc_aux_comb}.
Due to the bounds on $\eta$, this concludes the proof of the Caccioppoli estimate.
\end{proof}

\begin{proof}[Proof of Theorem \ref{thm:log_conv}]
The proof of Theorem \ref{thm:log_conv} from the Carleman estimate in Theorem \ref{thm:Carl} follows from a standard cut-off argument. For completeness, we present the details.

Let $u: (h\mathbb{Z})^d \rightarrow \mathbb{R}$ such that $P_h u(n) = 0$  for all  $n \in B_4$.
Fix $\varepsilon>0$ to be small enough and assume that $h_0>0$ is sufficiently small. We consider the function $w(n)=\theta(n)u(n)$, with $0\le \theta(x)\le 1$ a $C^{\infty}(\R^d)$ cut-off function defined as
$$
\theta(x)=
\begin{cases}
0, \quad &x\in B_{\frac14+\varepsilon}\cup B_{2-\varepsilon}^c\\
1, \quad &x\in B_{\frac32}\setminus B_{\frac12}.
\end{cases}
$$

Using the equation for $u$, we then write
\begin{align*}
&h^{-2}\D_d w(n)=\sum_{j=1}^d\big(\theta(n+he_j)u (n+he_j)+\theta(n-he_j)u(n-he_j)-2\theta(n)u(n)\big)h^{-2}\\
&=\theta(n)h^{-2}\D_du(n)+\sum_{j=1}^d\big((\theta(n+he_j)-\theta(n))u(n+he_j)+(\theta(n-he_j)-\theta(n))u(n-he_j)\big)h^{-2}\\
&=\theta(n)h^{-2}\D_du(n) + \sum_{j=1}^d\Big((\theta(n+he_j)-\theta(n))(u(n+he_j)-u(n-he_j))h^{-2}\\
&\qquad +(\theta(n-he_j)-\theta(n))u(n-he_j)+(\theta(n+he_j)-\theta(n))u(n-he_j)\Big)h^{-2}\\
&= \theta(n) V(n) u(n) + \theta(n) h^{-1}\sum\limits_{j=1}^{d} B_j(n) D^h_{+,j} u(n)\\
& \quad + \sum_{j=1}^d\big((\theta(n+he_j)-\theta(n))(u(n+he_j)-u(n-he_j))\big)h^{-2}+h^{-2}\D_d\theta(n)\sum_{j=1}^d u(n-he_j)\\
&=  V(n) w(n) + h^{-1}\sum\limits_{j=1}^{d} B_j(n) D^h_{+,j} w(n)\\
& \quad - h^{-1} \sum\limits_{j=1}^{d} B_j(n)(\theta(n+ h e_j) - \theta(n))u(n+he_j)) \\
& \quad + \sum_{j=1}^d\big((\theta(n+he_j)-\theta(n))(u(n+he_j)-u(n-he_j))\big)h^{-2}+h^{-2}\D_d\theta(n)\sum_{j=1}^d u(n-he_j)\\
&=:  V(n) w(n) + h^{-1}\sum\limits_{j=1}^{d} B_j(n) D^h_{+,j} w(n)
+ T_{d,1} u(n) + T_{d,2}u(n) + T_{d,3} u(n).
\end{align*}
Applying the Carleman estimate \eqref{eq:Carl_main} from Theorem \ref{thm:Carl}, using Remark \ref{rmk:deriv} and the triangle inequality, we obtain
\begin{equation}
\label{eq:apply_Carl}
\tau^{\frac{3}{2}}\|e^{\phi} w\| + \tau^{\frac{1}{2}} h^{-1} \|e^{\phi} D_+^hw\|  \leq C_{\text{Carl}}\Big( \| e^{\phi} V w\| +   h^{-1}   \|e^{\phi}B D^h_+ w\| + \sum\limits_{\ell=1}^3\| e^{\phi} T_{d,\ell} u\| \Big).
\end{equation}
Now choosing $\tau \geq 2C_{\text{Carl}}\max\{1,\|V\|_{L^{\infty}}^{\frac{2}{3}}, \|B\|_{L^{\infty}}^{2}\}$ allows us to absorb the first two contributions from the right hand side of \eqref{eq:apply_Carl} into the left hand side of \eqref{eq:apply_Carl}. We thus obtain the bound
\begin{equation}
\label{eq:apply_Carl1}
\tau^{\frac{3}{2}}\|e^{\phi} w\| + \tau^{\frac{1}{2}}  h^{-1}   \|e^{\phi} D_+^h w\| \leq 2C_{\text{Carl}} \sum\limits_{\ell=1}^3\| e^{\phi} T_{d,\ell} u\| .
\end{equation}
We next deal with the errors on the right hand side of \eqref{eq:apply_Carl1}.
On the one hand, we have for $j\in \{1,3\}$ 
$$
 \|e^{\phi}  T_{d,j} u\|^2
\leq C\big(\|e^{\phi}u\|^2_{B_{\frac12}\setminus B_{\frac14+\varepsilon}}+\|e^{\phi}u\|^2_{B_{2-\varepsilon}\setminus B_{\frac32}}\big)\le C\big(e^{2\tau \varphi(1/4+\varepsilon)}\|u\|^2_{B_{\frac12}}+e^{2\tau \varphi(3/2)}\|u\|^2_{B_{2}}\big).
$$
On the other hand, for $T_{d,2}$ 
\begin{equation*}
 \|e^{\phi}  T_{d,j} u\|^2
\leq C \big(\|e^{\phi} D_su\|_{B_{2-\varepsilon}\setminus B_{\frac32}}^2+\|e^{\phi}D_s u\|_{B_{\frac12}\setminus B_{{\frac14}+\varepsilon}}^2\big)\le C \big(e^{2\tau \varphi(3/2)} \|u\|_{B_2}^2+e^{2\tau\varphi(1/4+\varepsilon)} \|u\|_{B_{\frac12}}^2\big),
\end{equation*}
where we used the Caccioppoli estimate from Lemma \ref{lem:Caccioppoli}.

 Moreover, since $w\equiv u$ in $B_{3/2}\setminus B_{1/2}$, we have
$$
\tau^3 \|e^{\phi} w   \|^2\ge \tau^3\|e^{\phi}u\|_{B_1\setminus B_{\frac12}}^2\ge  \tau^3e^{2\tau\varphi(1)}\| u\|_{B_1\setminus B_{\frac12}}^2.
$$
In view of the above, we get
$$
\tau^3e^{2\tau\varphi(1)}\| u\|_{B_1\setminus B_{\frac12}}^2\le  C \big(e^{2\tau\varphi(3/2)} \|u\|_{B_2}^2+e^{2\tau\varphi(1/4+\varepsilon)} \|u\|_{B_{\frac12}}^2\big),
$$
and since $\varphi$ is decreasing,
\begin{align*}
\| u\|_{B_1\setminus B_{\frac12}}^2
& \le C \big(\tau^{-3}e^{2\tau\varphi(3/2)-2\tau \varphi(1)} \|u\|_{B_2}^2+\tau^{-3}e^{2\tau\varphi(1/4)-2\tau\varphi(1)} \|u\|_{B_{\frac12}}^2\big)\\
& \le C\big(e^{-2c_2 \tau} \|u\|_{B_{2}}^2+e^{2c_1\tau} \|u\|_{B_{\frac12}}^2\big)
\end{align*}
for some constants $c_1, c_2>0$ with $c_1:=|\varphi(3/2)- \varphi(1)|$ and $c_2:=\varphi(1/4)-\varphi(1)>0$ (for which we choose the constant $c_{ps}>0$ in Theorem \ref{thm:Carl} and Lemma \ref{lem:weight} sufficiently small). Since further trivially $\| u\|_{B_{\frac12}}^2\le e^{2c_1\tau} \|u\|_{B_{\frac12}}^2$, this concludes the proof.
\end{proof}

\begin{rmk}
\label{rmk:singular_potentials}
We remark that as a feature of the discrete setting, to a certain degree we can also deal with more singular potentials. Tracking the argument from above (in particular the passage from
\eqref{eq:apply_Carl} to \eqref{eq:apply_Carl1}), we note that if $V$ and $B$ only satisfy the bounds
\begin{equation*}
\|V\|_{L^{\infty}(B_{4})} \leq \mu_0 h^{-\frac{3}{2}},\qquad \| B\|_{L^{\infty}(B_4)} \leq \mu_0 h^{-\frac{1}{2}},
\end{equation*}
with $\mu_0 \leq \frac{C_{\textup{Carl}}}{10} \delta_0$, we can deduce that for some constants $\tilde{c}_1, \tilde{c}_2>0$ (independent of $h$)
\begin{equation*}
\|u\|_{L(B_1)} \leq C (e^{ \tilde{c}_1 h^{-1}} \|u\|_{L^2(B_{\frac{1}{2}})} + e^{- \tilde{c}_2 h^{-1}} \|u\|_{L^2(B_2)}).
\end{equation*}
We also remark that while yielding quantitative propagation of smallness type estimates, as expected these estimates do not pass to the limit $h \rightarrow 0$. Further, the $h$ dependence in the exponentials can be adapted to the size of the potentials (with different bounds in the exponents of the logarithmic convexity estimates depending on the bounds on $V$, $B$).
\end{rmk}

\section{Remarks on Scaling}
\label{sec:scaling}

Having established \eqref{eq:3_balls}, we note that to a certain degree -- although this is substantially weaker than in the continuous setting -- it is possible to rescale this estimate. We discuss this in the case of the Laplacian (for more general operators similar observations remain valid). To this end, we make the following observation. We shall use the notation $\D_{d,h} =h^{-2}\Delta_d$. 

\begin{lem}
\label{lem:rescaling_discrete}
Let $u:B_4 \rightarrow \R$ be such that $\D_{d,h} u = 0$ in $B_R \subset (h\Z)^d$. Then, for any $m \in \N$ such that $h m \leq 2$, we also have $\D_{d,mh} u =0$ in $B_{R/m} \subset (mh \Z)^d$ (i.e. with respect to the lattice $(mh \Z)^d$).
\end{lem}

\begin{proof}
We prove the statement inductively in $m$. For the case $m=2$ we have to show that
\begin{equation*}
\sum\limits_{j=1}^d( u(x + h e_j) + u(x - he_j)-2u(x)) =0\quad  \mbox{ for } \quad  x \in (h \Z)^d
\end{equation*}
implies that
\begin{equation*}
\sum\limits_{j=1}^d( u(x + 2h e_j) + u(x - 2he_j)-2u(x)) =0 \quad  \mbox{ for } \quad  x \in ( 2 h \Z)^d.
\end{equation*}
In order to observe this, we note that
\begin{align*}
u(x + 2h e_j) + u(x - 2he_j)-2u(x)
&= (u(x + 2h e_j) + u(x) -2 u(x+ he_j)) \\
& \quad + 2(u(x+he_j)+ u(x-he_j)-2 u(x))\\
& \quad  + (u(x-2h e_j) + u(x) - 2u(x-he_j)).
\end{align*}
Summing and noting that the corresponding contributions in the brackets yield the Laplacian on $(h\Z)^d$ implies the claim for $m=2$.

Assuming the induction hypothesis for any $m$, i.e.,
\begin{equation*}
\sum\limits_{j=1}^d( u(x + mh e_j) + u(x - mhe_j)-2u(x)) =0 \quad \mbox{ for } \quad x \in ( m h \Z)^d,
\end{equation*}
we prove the statement for $m+1$. We have
\begin{align*}
&u(x + (m+1)h e_j) + u(x - (m+1)he_j)-2u(x)\\
&\,\, = (u(x + (m+1)h e_j) + u(x+(m-1)he_j)-2 u(x+ mhe_j))\\
&\quad -\big(u(x+(m-1)he_j)+u(x-(m-1)he_j) -2u(x)\big) + 2\big(u(x+mhe_j)+ u(x-mhe_j)-2 u(x)\big)\\
& \quad + (u(x-(m+1)h e_j) + u(x-(m-1)he_j) - 2u(x-mhe_j)).
\end{align*}
The conclusion follows from the cases $m=1$ (after translation) and the inductive steps for $m$ and $m-1$.

\end{proof}

Using the previous auxiliary result, we may infer rescaled versions of Theorem \ref{thm:3spheres}:

\begin{cor}
\label{cor:rescale}
Let $u: (h\Z)^d \rightarrow \R$ be such that $\D_{d,h} u = 0$ in $B_R$. Assume that $u: (m^{-1} h \Z) \rightarrow \R$ is also such that $\D_{d,m^{-1}h} u = 0$.
Then there exist $\alpha \in (0,1)$, $c_0>0$ $h_0 >0$ and $C>1$ (independent of $u$) such that for $h\in (0,h_0)$
\begin{equation*}
\|u\|_{L^2(B_{m^{-1}})} \leq C(\|u\|_{L^2(B_{m^{-1}/2})}^{\alpha}\|u\|_{L^2(B_{2m^{-1}})}^{1-\alpha} + 2^{ -c_0 h^{-1}}\|u\|_{L^2(B_{2m^{-1}})}).
\end{equation*}
\end{cor}

\begin{proof}
We consider the function $u_m(x):= u(m^{-1}x)$ with $x \in (h\Z)^d$. By the considerations from Lemma \ref{lem:rescaling_discrete} this is also harmonic on $(h\Z)^d$. Thus, we may apply Theorem \ref{thm:3spheres}. Rescaling $z=m^{-1}x$ then implies the claim.
\end{proof}

\begin{rmk}
We remark that, of course, apart from rescalings also translations are always possible due to the translation invariance of the operator at hand.
\end{rmk}

\section*{Acknowledgements}

The first author is supported by ERCEA Advanced Grant 2014 669689 - HADE, by the project PGC2018-094528-B-I00 (AEI/FEDER, UE) and acronym ``IHAIP'', and by the Basque Government through the project IT1247-19. The second author is supported by the Basque Government through the BERC 2018-2021 program, by the Spanish Ministry of Economy and Competitiveness MINECO: BCAM Severo Ochoa excellence accreditation SEV-2017-2018  and through project MTM2017-82160-C2-1-P funded by (AEI/FEDER, UE) and acronym ``HAQMEC''. She also acknowledges the RyC project RYC2018-025477-I and IKERBASQUE.
The fourth author is supported by the Spanish research project PGC2018-094522 N-100 from the MICINNU.

The authors would like to thank Sylvain Ervedoza for pointing out the optimal scaling in $\tau h \leq \delta_0$ in the Carleman inequality.


\end{document}